\RequirePackage{fix-cm}
\documentclass[smallextended]{svjour3}       % onecolumn (second format)
\smartqed  % flush right qed marks, e.g. at end of proof
\usepackage{amssymb,amsmath,multirow}
\usepackage{xcolor}

\newtheorem{thm}{Theorem}%[section]
\newtheorem{prop}{Proposition}%[section]
\newtheorem{lem}{Lemma}%[section]
\newtheorem{cor}{Corollary}%[section]
\newtheorem{ass}{Assumption}%[section]
\newtheorem{exam}{Example}%[section]
\begin{document}
\title{S-Lemma with Equality and Its Applications
%On Semidefinite Programming Relaxations of the Unbalanced Orthogonal Procrustes Problem
\thanks{This research was supported
by Taiwan National Science Council under grant 102-2115-M-006-010,
by National Center for Theoretical Sciences (South), by National Natural Science Foundation of China under grant
11471325 and by Beijing Higher Education Young Elite Teacher Project 29201442.
}
}

\titlerunning{S-Lemma with Equality }        % if too long for running head

\author{ Yong Xia  \and  Shu Wang \and Ruey-Lin Sheu}


\institute{
Y. Xia \at
              State Key Laboratory of Software Development
              Environment, LMIB of the Ministry of Education,
              School of
Mathematics and System Sciences, Beihang University, Beijing,
100191, P. R. China
              \email{dearyxia@gmail.com}           %  \\
\and
 S. Wang         \at
              School of
Mathematics and System Sciences, Beihang University, Beijing,
100191, P. R. China \email{wangshu.0130@163.com }
\and
R. L. Sheu \at
              Department of Mathematics, National Cheng Kung University, Taiwan
              \email{rsheu@mail.ncku.edu.tw}
 }

\date{Received: date / Accepted: date}

\maketitle

\begin{abstract}
Let $f(x)=x^TAx+2a^Tx+c$ and $h(x)=x^TBx+2b^Tx+d$ be two quadratic
functions having symmetric matrices $A$ and $B$. The S-lemma with
equality asks when the unsolvability of the system $f(x)<0, h(x)=0$
implies the existence of a real number $\mu$ such that $f(x) + \mu
h(x)\ge0, ~\forall x\in \mathbb{R}^n$. The problem is much harder
than the inequality version which asserts that, under Slater
condition, $f(x)<0, h(x)\le0$ is unsolvable if and only if $f(x) +
\mu h(x)\ge0, ~\forall x\in \mathbb{R}^n$ for some $\mu\ge0$. In
this paper, we show that the S-lemma with equality does not hold
only when the matrix $A$ has exactly one negative eigenvalue and
$h(x)$ is a non-constant linear function ($B=0, b\not=0$).
%we show that the S-lemma with equality must hold unless
%the matrix $A$ has exactly one negative eigenvalue; $h(x)$ is a
%non-constant affine function ($B=0, b\not=0$); and the matrix in
%(\ref{m}) is positive semi-definite.
%We show that the S-lemma with equality is always true
%unless the matrix $A$ has exactly one negative eigenvalue; $h(x)$ is
%a non-constant linear function ($B=0, b\not=0$); and one natural
%relation between $A$ and $b$ is met (Theorem 2).
As an application, we can globally solve $\inf\{f(x): h(x)=0\}$
as well as the two-sided generalized trust region subproblem
$\inf\{f(x): l\le h(x)\le u\}$ without any condition. Moreover,
the convexity of the joint numerical range $\{(f(x), h_1(x),\ldots,
h_p(x)):x\in\Bbb R^n\}$ where $f$ is a (possibly non-convex)
quadratic function and $h_1(x),\ldots,h_p(x)$ are affine functions
can be characterized using the newly developed S-lemma with
equality.
%$f(x)<0, h(x)=0$ is
%unsolvable with

\end{abstract}

\section{Introduction}

Let $f(x)=x^TAx+2a^Tx+c$ and $h(x)=x^TBx+2b^Tx+d$ be two quadratic
functions having symmetric matrices $A$ and $B$. In 1971, Yakubovich
\cite{y71,y77} proved a fundamental result, which we call the {\it
classical S-lemma} in this paper. It asserts that, given any pair of
quadratic functions $(f,h)$, if $h(x)\le0$ satisfies Slater's
condition, namely, there is an $\overline{x}\in \mathbb{R}^n$ such
that $h(\overline{x}) < 0$, the following two statements are always
equivalent:
\begin{itemize}
 \item[](${\rm S_1}$)~~ ($\forall x\in\Bbb R^n$) $~h(x)\le 0~\Longrightarrow~ f(x)\ge 0.$ %\label{slm}
 \item[](${\rm S_2}$)~~   There exists a $\mu\ge0$ such that
$f(x) + \mu h(x)\ge0, ~\forall x\in \mathbb{R}^n.$ %\label{lm}
\end{itemize}
In the following, let us denote the equivalence by (${\rm
S_1}$)$\sim$(${\rm S_2}$) to mean that both statements are true or
false synchronously.

Notice that (${\rm S_2}$) trivially implies (${\rm S_1}$), but the
other way around from (${\rm S_1}$) to (${\rm S_2}$) is not so
obvious. Yakubovich's proof (see also \cite{PT}) had to rely on (i)
a homogenization scheme which transforms the nonhomogeneous system
of (${\rm S_1}$) and (${\rm S_2}$) into homogeneous ones; (ii) if
$f$ and $h$ are quadratic forms, the joint numerical range $\{(f(x),
h(x)): x\in\Bbb R^n\}$ is convex \cite{Dines}; and (iii) Slater's
condition for applying the separation theorem to provide the
existence of $\mu\ge0$.

The classical S-lemma is a powerful tool especially in control
theory and robust optimization. See recent surveys in \cite{D06,PT}.
It is a form of the celebrated Farkas lemma \cite{J00}. It can be
applied to solve and show the hidden convexity of the quadratic
programming with a single quadratic inequality constraint (QP1QC)
under Slater's condition. The connection is easily illustrated by
\begin{align}
v({\rm QP1QC})&=\inf_{x\in \Bbb R^n}\left\{f(x): h(x)\le0\right\}\nonumber\\
&=\sup\limits_{\lambda\in\Bbb R}\left\{\lambda:\Big\{x\in \Bbb
R^n|f(x)<\lambda,~ h(x)\le0\Big\}=\emptyset\right\}\nonumber\\
&=\sup\limits_{\lambda\in\Bbb R}\left\{\lambda: (\exists\mu\ge0)~f(x)-\lambda+\mu h(x)\ge0,\forall x\in \Bbb R^n\right\}\label{qp1qc-1}\\
&=\sup_{\lambda\in\Bbb R,\mu\ge0}
\left\{\lambda:\left[\begin{array}{cc}
A+\mu B& a+\mu b\\
a^T+\mu b^T& c+\mu d -\lambda
\end{array}\right]
\succeq0\right\}\label{LMI}
\end{align}
where $v(\cdot)$ is the optimal value of problem $(\cdot)$ and the
notation $X\succeq0$ implies that $X$ is positive semidefinite.
Notice that the key step (\ref{qp1qc-1}) is due to (${\rm
S_1}$)$\sim$(${\rm S_2}$), and the SDP result in (\ref{LMI}) shows
the hidden convexity of (QP1QC).

The classical S-lemma has several forms of generalization. Jeyakuma
et al. in \cite{J09b} extended $\Bbb R^n$ in (${\rm S_1}$) and
(${\rm S_2}$) to any linear manifold. The result was applied to
characterize the global optimality of (QP1QC) equipped with
additional linear equality constraints. When three quadratic
functions are considered, under the assumption that there exists a
positive definite matrix pencil of the three quadratic functions,
Jeyakuma et al. in \cite{J09b} also gave a type of alternative
theorem involving only strict inequalities. Generalizations of the
classical S-lemma having two or more different $h$'s are referred to
as {\it S-procedure}. See \cite{D06,PT} for a survey. Particularly,
Fradkov and Yakubovich \cite{Fr} proved that the strong duality
holds for nonconvex quadratic optimization with two quadratic
constraints in {\it complex} variables \cite{BE}. Jeyakumar et al.
in \cite{J09a} proved that, in the absence of Slater's condition but
assuming that $\{x: h(x) \le0\}\not=\emptyset$, (${\rm
S_1}$)$\sim$(${\rm S_2}$) cannot be true for all $f$'s. However,
$(f,g)$ still satisfies a weaker equivalence (${\rm
S_1}$)$\sim$(${\rm S^r_2}$), the so-called {\it regularized S-lemma}
below:
\begin{itemize}
\item[](${\rm S_1}$)~~ ($\forall x\in\Bbb R^n$) $~h(x)\le 0~\Longrightarrow~ f(x)\ge 0.$ %\label{slm}
\item[](${\rm S^r_2}$)~~ ($\forall \epsilon>0$) ($\exists \lambda_{\epsilon} \ge 0$)
($\forall x\in\Bbb R^n$) $~f(x)+ \lambda_{\epsilon}h(x)+
\epsilon(x^Tx + 1) \ge 0.$
\end{itemize}
It is weaker because (${\rm S_2}$) implies (${\rm S^r_2}$) by
letting $\lambda_{\epsilon}=\mu\ge0,~\forall \epsilon>0$.
%
%\begin{equation}
%h(x) \le0 ~\Longrightarrow~ f (x) \ge 0, ~\forall x\in \mathbb{R}^n
%\label{reg:in1}
%\end{equation}
%holds if and only if
%\begin{equation}
%(\forall \epsilon>0) (\exists \lambda_{\epsilon} \ge 0) (\forall x\in\Bbb R^n)
%~f(x)+ \lambda_{\epsilon}h(x)+ \epsilon(x^Tx + 1) \ge 0.\label{reg:in2}
%\end{equation}

This paper studies the {\it S-lemma with equality}. In the formality,
it replaces $h\le0$ in (${\rm S_1}$) by $h=0$ and $\mu\ge0$ in
(${\rm S_2}$) by $\mu\in \mathbb{R}$. It asks, for what pairs of
$(f,h)$ the following two statements are equivalent (i.e. (${\rm
E_1}$)$\sim$(${\rm E_2}$)):
\begin{itemize}
\item[](${\rm E_1}$)~~ ($\forall x\in\Bbb R^n$) $~h(x)= 0~\Longrightarrow~ f(x)\ge 0.$ %\label{slm}
\item[](${\rm E_2}$)~~ There exists a $\mu\in \mathbb{R}$ such that
$f(x) + \mu h(x)\ge0, ~\forall x\in \mathbb{R}^n.$ %\label{lm}
\end{itemize}
With the replacement, (${\rm S_1}$) is relaxed to (${\rm E_1}$)
whereas (${\rm S_2}$) is relaxed to (${\rm E_2}$). Though it is easy
to see that (${\rm E_1}$)$\sim$(${\rm E_2}$) does not always hold,
yet it is by no means trivial to characterize conditions that
$(f,h)$ should satisfy to make (${\rm E_1}$)$\sim$(${\rm E_2}$)
correct. This is going to be the main theme of this paper.

To begin with the discussion, we always assume that $ \{x:
h(x)=0\}\neq \emptyset.$ In literature, the first variant of the
S-lemma with equality was proposed by Finsler \cite{F} in 1937 for
the homogeneous system. Different proofs of Finsler's Theorem can be
found in \cite{H75,M93}. It states that the statement
\begin{equation}
 %x^TAx \le 0,~x^TBx = 0,~x\neq 0
(\forall x\in\Bbb R^n,~x\not=0)~~x^TBx = 0 ~\Longrightarrow~ x^TAx> 0
 \label{Finsler}
\end{equation}
is true if and only if there exists a $\mu\in \Bbb R$ such that $A +
\mu B\succ0$. However, from the simple example that
$f(x_1,x_2)=x_1x_2,~h(x_1,x_2)=-x_1^2$, we see that
$h(x_1,x_2)=-x_1^2=0$ implies that $f(x_1,x_2)=0$ so (${\rm E_1}$)
is true but (\ref{Finsler}) is not valid. Finsler's Theorem then
asserts that there is no $\mu\in \Bbb R$ such that $A + \mu
B\succ0$, but, indeed, there is no $\mu\in \Bbb R$ such that $f(x) +
\mu h(x)=x_1x_2-\mu x^2_1\ge0, ~\forall x\in \mathbb{R}^n.$ Thus
(${\rm E_2}$) is false and (${\rm E_1}$)$\not\sim$(${\rm E_2}$). The
example shows that the equivalence of (${\rm E_1}$) and (${\rm
E_2}$) is in general not true, even just for a simple homogeneous
system by a slight generalization from Finsler's Theorem. We notice
that in this example $h(x_1,x_2)=-x_1^2\le0$ satisfies Slater's
condition so that (${\rm S_1}$)$\sim$(${\rm S_2}$). It does not
satisfy the following ``two-side'' Slater's condition though.
\begin{ass}\label{as1}
$h(x)$ takes both positive and negative values, i.e., there are
$x',x''\in \Bbb R^n$ such that $h(x')<0<h(x'')$.
\end{ass}

Assumption \ref{as1} is obviously stricter than the usual Slater's
condition. However, even when it is imposed, (${\rm
E_1}$)$\sim$(${\rm E_2}$) may still be invalid. For example, let
$f(x_1,x_2)=x_1^2-x_2^2,~h(x_1,x_2)=x_2$ where $h(x_1,x_2)=x_2$
satisfies Assumption \ref{as1}. We can check that $h(x_1,x_2)=x_2=0$
implies that $f(x_1,0)=x_1^2\ge0$, but there is still no $\mu\in
\Bbb R$ such that $f+\mu h=x_1^2-x_2^2+\mu x_2\ge0,~\forall
x_1,x_2\in \mathbb{R}$. So (${\rm E_1}$)$\not\sim$(${\rm E_2}$).
However, since $h$ satisfies Slater's condition, there must be
(${\rm S_1}$)$\sim$(${\rm S_2}$). A simple verification shows that
``$h(x_1,x_2)=x_2\le0 \nRightarrow f(x_1,x_2)=x_1^2-x_2^2\ge0$'' so
that both (${\rm S_1}$) and (${\rm S_2}$) are false in this case.

There were several attempts trying to establish some affirmative
results for (${\rm E_1}$)$\sim$(${\rm E_2}$), but they all came with
an incomplete sufficient condition. Analogous to the role of
Slater's condition in (${\rm S_1}$)$\sim$(${\rm S_2}$), all
sufficient conditions for (${\rm E_1}$)$\sim$(${\rm E_2}$) are
subject to Assumption \ref{as1}. In literature, those sufficient
conditions are \vskip0.1cm

\noindent{\bf S-Condition 1}(\cite{PT}): $h(x)$ is strictly concave
(or strictly convex). \vskip0.1cm

\noindent{\bf S-Condition 2}(\cite{BE}, Thm. A.2): There is an
$\eta\in \Bbb R$ such that $ A \succeq\eta B.$ \vskip0.1cm

\noindent{\bf S-Condition 3}(\cite{TT}, Corollary 6): $h(x)$ is
homogeneous. \vskip0.1cm

\noindent{\bf S-Condition 4}(\cite{Bong}): $h(0)=0$ and there exists
$\zeta\in X=\{x\in \Bbb R^n: h(x)=0\}$ such that
\begin{equation}
(\forall x\in \mathbb{R}^n)~~x^TBx=0 ~\Longrightarrow ~(B\zeta+b)^Tx=0.
\label{as55}
\end{equation}
We defer the discussion about the relations among S-Conditions 1 - 4
to Appendix for readers who are interested, but consider the
following example, which shows that none of S-Conditions 1 - 4 can
become necessary.

Let $f(x_1,x_2)=-x_1^2-x_2^2,~h(x_1,x_2)=x_2$. We check that
$h(x_1,x_2)=x_2$ satisfies Assumption \ref{as1}. For (${\rm E_1}$),
$h(x_1,x_2)=x_2=0$ does not imply that $f(x_1,0)=-x_1^2\ge0,~\forall
x_1\in\mathbb{R}$. For (${\rm E_2}$), there is no $\mu\in \Bbb R$ such
that $f+\mu h=-x_1^2-x_2^2+\mu x_2\ge0,~\forall x_1,x_2\in
\mathbb{R}$. Since both (${\rm E_1}$) and (${\rm E_2}$) are false,
(${\rm E_1}$)$\sim$(${\rm E_2}$) in this example. However, this
affirmative example cannot be characterized by either S-Condition
above. We can verify that $h$ is not strictly convex (concave);
since $B=0,~A\prec0$, there is no $\eta\in \Bbb R$ such that $ A
\succeq\eta B; $ $h=0$ is a hyperplane so it is not homogeneous;
finally $h(0)=0$, but due to $B=0,~b\not=0$, S-Condition 4 does not
hold.

From the above discussion, we learn that (${\rm E_1}$)$\sim$(${\rm
E_2}$) is a much harder problem than the classical S-lemma (${\rm
S_1}$)$\sim$(${\rm S_2}$). No attempt has been successfully made so
far. Our approach is to beak down the problem into several cases,
each of which encompasses a unique algebraic and/or geometrical
feature of $(f,h)$ for easy analysis. We found that, when $h(x)$
takes both positive and negative values, the statement (${\rm E_1}$)
(i.e. $\inf_{h(x)=0}f(x)\ge0$ holds) {\it nearly implies that} there
is a scalar $\mu$ to adjust the size and/or the direction of $h$
such that the linear combination $f(x)+\mu h(x)$ becomes convex with
an attainable minimum (on the entire $\mathbb{R}^n$, not just on
$h(x)=0$). The only exception is when $h(x)=0$ is a ``flat'' $n-1$
dimensional hyperplane on which $f$ is convex while on the remaining
dimension (complement to $h(x)=0$) $f$ becomes concave. If described
in the algebraic way, under Assumption \ref{as1}, if $B\not=0$,
there must be (${\rm E_1}$)$\sim$(${\rm E_2}$). The only type of
examples that can have (${\rm E_1}$)$\not\sim$(${\rm E_2}$) is when,
and only when $B=0,~b\not=0$, the matrix $A$ has exactly one
negative eigenvalue, and the matrix in (\ref{m}) is positive
semi-definite. The result explains why
$f(x_1,x_2)=-x_1^2-x_2^2,~h(x_1,x_2)=x_2$ gets an affirmative result
(the matrix $A$ has two negative eigenvalues), while
$f(x_1,x_2)=x_1^2-x_2^2,~h(x_1,x_2)=x_2$ gets a negative assertion
($B=0$, $b^T=(0,1/2)$, $A$ has exactly one negative eigenvalue, and
the matrix in (\ref{m}) is %$\left[\begin{array}{cc}1&0\\0&0\end{array}\right]\succeq0$). 
${\rm Diag}(1,0)\succeq0$). 
Theorem
\ref{thm:m} in Section 3 gives the complete statement as well as the
proof for the S-lemma with equality to hold under Assumption
\ref{as1}, whereas the characterization of the theorem when
Assumption \ref{as1} fails will be treated in Section 2.

Now we turn to the applications of the S-lemma with equality. The
importance of the S-lemma with equality was not understood by us
until we eventually realized from \cite{Bong} that it is the key to
solve the following long-standing interval bounded generalized trust
region subproblem (studied in \cite{Pong} by Pong and Wolkowicz and
also by many other researchers)
\begin{eqnarray}
({\rm GTRS})~~~&\inf& f(x)\nonumber \\
&{\rm s.t.}& l\le h(x)\le u. \label{interval-bound}
\end{eqnarray}
If $f(x)$ is convex and the optimal solution to the unconstraint
problem $\min\limits_{x\in \mathbb{R}^n} f(x)$ happens to be
feasible to (\ref{interval-bound}) (by checking whether $\nabla
f(x)=0,~l\le h(x)\le u$ has a solution), then the optimal value
$\displaystyle v({\rm GTRS})=\min\limits_{x\in \mathbb{R}^n} f(x)$.
Otherwise, the optimal solution is located at one of the two
boundaries:
\[
v({\rm GTRS})=\min\left\{\inf_{h(x)=l} f(x),~\inf_{h(x)=u} f(x) \right\}.
\]
It reduces (GTRS) to
\begin{eqnarray*}
({\rm QP1EQC})~~~&\inf& f(x)\\ %\label{qp1}\\
&{\rm s.t.}& h(x)=0.%\label{qp2}
\end{eqnarray*}
With the same idea for solving (QP1QC) by applying the classical
S-lemma in (\ref{qp1qc-1})-(\ref{LMI}), (QP1EQC) can be similarly
proved to have the strong duality by the S-lemma with equality. We
also discuss when (QP1EQC) becomes unbounded below, and give a
necessary and sufficient condition to characterize when (QP1EQC) is
attainable. As a consequence, both (QP1EQC) and (GTRS) are
completely analyzed.

We remark that (QP1EQC) itself has many interesting applications,
including the double well potential optimization problem
\cite{Fa,Fa-2}, the time of arrival geolocation problem \cite{Hm}
and unbiased least squares optimization for system identification
\cite{Pa}. In particular, the double well potential model came from
numerical approximations to the generalized Ginzburg-Landau
functionals \cite{Jerrard}. It minimizes the following special type
of multi-variate polynomial of degree 4:
\[
\min_{x \in \Bbb R^n} \frac{1}{2}\left(\frac{1}{2} \|Qx-c\|^2-d \right)^2
+\frac{1}{2}x^TAx-a^Tx
\]
where $Q\not=0$ is an $m\times n$ matrix. We can reformulate it as
an example of (QP1EQC):
\begin{eqnarray*}
({\rm DWP})~~~&\min_{x \in \Bbb R^n,z\in \Bbb R}& \frac{1}{2}z^2+\frac{1}{2}x^TAx-a^Tx\\
&{\rm s.t.}& \frac{1}{2} \|Qx-c\|^2-d-z=0.
\end{eqnarray*}
In the constraint function, the variable $z$ does not have a second
order term, so it is not strictly concave (or strictly convex). When
$A$ has a negative eigenvalue whose eigenvector is not in the range
of $Q^TQ$, there is no $\eta\in \Bbb R$ such that $A\succeq\eta
Q^TQ.$ Moreover, (DWP) is in general non-homogeneous. Finally, the
model obviously fails an equivalent statement of S-Condition 4 in
(\ref{bB}), which we derive in Appendix. In other words, none of the
existing results leads a complete answer to (DWP).
%
%
%to resolve (DWP) in general, it also requires to understand (QP1QC)
%completely.

 %Other examples include
 %Variations of ({\rm QP1EQC}) have been
%extensively studied in the literature, see for example,
%\cite{Ben14,Ben,Feng,M93,Pa}. While it is known that the inequality
%version $\inf\{f(x)\vert h(x)\le0\}$ can be solved by an SDP
%relaxation with a rank-one decomposition procedure (e.g.
%\cite{sz,yz}), the equality version ({\rm QP1EQC}) cannot be
%similarly done without various conditions. Moreover, since an
%optimal solution to ({\rm QP1EQC}) even may not be a local minimum
%for the inequality version, the latter does not help much for
%solving the former. However, if we can understand completely the
%S-lemma with equality, it will then be easy to derive the strong
%duality and a tight SDP relaxation for ({\rm QP1EQC}). A rank-one
%decomposition procedure can also apply immediately.

Finally in Section 5, by applying the S-lemma with equality, we
obtain a {\it necessary and sufficient} description for the
convexity of the joint numerical range
$S=\{(f(x),h_1(x),\ldots,h_p(x)): x\in\Bbb R^n\}$ where
$h_1,\ldots,h_p$ are all affine functions. Dines in 1941
\cite{Dines} showed that the joint numerical range $M= \{(f(x),
h(x)):x\in \Bbb R^n\}$ is a convex subset in $\Bbb R^2$ when $f$ and
$h$ are quadratic forms. Dines' theorem later became a major
machinery for Yakubovich to prove the famous classical S-lemma in
1971. However, it has been shown by Polyak in 1998 \cite{Poly} that
Dines' result is in general not true even when one of $f$ and $h$ is
affine. Beck \cite{BE07} has shown that, when $f$ is strictly convex
and $p\le n-1$, $S$ is closed and convex, but $S$ is non-convex when
$p=n$ and $h_1,h_2,\ldots,h_n$ have linearly independent normals.
Our necessary and sufficient conditions for $S$ being convex cover
Polyak's counterexample and Beck's result as special cases.

Below we highlight a list of main contributions in the paper.
\begin{itemize}
\item When Assumption \ref{as1} fails, Theorem \ref{new1} characterizes the necessary
and sufficient conditions under which ${\rm (E_1)}$ $\sim$ ${\rm
(E_2)}$. (Sect. 2)
\item When Assumption \ref{as1} fails, Theorem \ref{new1:2} gives a ``regularized
S-lemma with equality'' ${\rm (E_1)}\sim{\rm (E^r_2)}$. The
classical regularized S-lemma (${\rm S_1}$)$\sim$(${\rm S^r_2}$)
by Jeyakumar et al. in \cite{J09a} is shown to be a direct
consequence of ${\rm (E_1)}\sim{\rm (E^r_2)}$. (Sect. 2)
\item When Assumption \ref{as1} holds, Theorem \ref{thm:m} characterizes the necessary
and sufficient conditions under which ${\rm (E_1)}$ $\sim$ ${\rm (E_2)}$. (Sect. 3)
\item As an application of Theorem \ref{thm:m}, under Slater's condition,
the classical S-lemma (${\rm S_1}$)$\sim$(${\rm S_2}$) is shown
to be a direct consequence of ${\rm (E_1)}$ $\sim$ ${\rm (E_2)}$. (Sect. 3)
\item As an application of Theorem \ref{thm:m}, problems (QP1EQC) as well as (GTRS) are
completely solved without any condition. (Sect. 4)
\item As an application of Theorem \ref{thm:m}, Beck's result \cite{BE07}
about the convexity of %the joint numerical range
$S=\{(f(x),h_1(x),\ldots,h_p(x)): x\in\Bbb R^n\}$ with $f$ being strictly
convex and $h_i's$ being affine can be now generalized to any %non-strictly-convex or
quadratic function $f$ in Theorem \ref{thm:cv}. (Sect. 5)
\item The most general version of the ``regularized S-lemma with equality'' is given in
Corollary 1 without any condition. (Sect. 3)
\end{itemize}

%
%
%
%it is interesting to note that whereas we can now say something
%reversely from the S-lemma to the joint numerical range.

Throughout the paper, we always assume that $\{x\in\Bbb R^n:
h(x)=0\}\not=\emptyset$. Notation $A\succeq (\preceq) B$ denotes
that the matrix $A-B$ is positive (negative) semidefinite. $A\succ
(\prec) B$ means that the matrix $A-B$ is positive (negative)
definite. $S_+^n$ represents the space of all $n\times n$ positive
semidefinite symmetric matrices. $A\bullet B={\rm
Tr}(AB^T)=\sum_{i,j=1}^na_{ij}b_{ij}$ stands for the standard inner
product of two symmetric matrices $A, B$. The null and range space
of $B$ is denoted by $\mathcal{N}(B)$ and $\mathcal{R}(B)$,
respectively; and $B^+$ is the Moore-Penrose generalized inverse of
$B$. Denote by $I_n$ the identity matrix of dimension $n$; and by
${\rm Diag}(a)$ the diagonal matrix with $a$ being its diagonal
vector. The notation $v(\cdot)$ denotes the optimal value of a
particularly mentioned optimization problem $(\cdot)$.

%
%The true complication comes from improving Assumption \ref{as2}. In
%literature, there are at least three generalizations with efforts to
%relax the strict convexity of $h$ (all under Assumption \ref{as1}
%that $h$ takes both positive and negative values):
%
%
%
%In Section 3, we adopt the commonly used homogenization technique by
%introducing a new variable $t$, but pay much attention to the case
%$t=0$ where the difficulty arises. Through a case-by-case analysis,
%we show that, under Assumption \ref{as1}, the S-lemma with equality
%always holds unless the matrix $A$ has exactly one negative
%eigenvalue, $B=0, b\not=0$, and
%\begin{equation*}
%\left[\begin{array}{cc}V^TAV&V^T(Ax_0+a)\\
%(x_0^TA+a^T)V&f(x_0)\end{array}\right]\succeq 0,
%\end{equation*}
%where $x_0=-\frac{d}{2b^Tb}b$, $V\in \Bbb R^{n\times (n-1)}$ is the
%matrix basis of $\mathcal{N}(b)$.
%

%\textcolor{red}{
%The above result indeed generalizes the S-lemma with inequality
%(\ref{slm})-(\ref{lm}), a very short proof can be found at the end of
%Section 3. Moreover, as a further extension, we will establish the
%regularized S-lemma with equality, which generalizes the regularized
%S-lemma due to Jeyakumar et al. \cite{J09a}.
%}

\section{The S-lemma with equality when Assumption \ref{as1} fails }

We observe that Assumption \ref{as1} is violated and $\{x\in\Bbb
R^n: h(x)=0\}\not=\emptyset$ if and only if
\[
\min_x h(x)= 0~ {\rm or}~ \max_x h(x)=0.
\]
Namely, $h$ is convex(concave) and the set $\{x:\,h(x)=0\}$ consists
of all the minimizers (maximizers) of $h(x)$ such that
\begin{equation}
\{x:  h(x)=0\}=\{x: \min_x(\max_x)h(x)=0\}=\{-B^+b+Zy: y\in\Bbb R^m\}\label{h0}
\end{equation}
where $Z\in\Bbb R^{n\times m}$ is a matrix basis of
$\mathcal{N}(B)$, and
\begin{eqnarray*}
\min_x(\max_x)h(x)&=&(-B^+b+Zy)^TB(-B^+b+Zy)+2b^T(-B^+b+Zy)+d\\
&=& -b^TB^+b+d=0.
\end{eqnarray*}
We then investigate (${\rm E_1}$)$\sim$(${\rm E_2}$) under the
following property:
\begin{prop}\label{prop1}
Assumption \ref{as1} is violated if and only if
\begin{equation}
 B\succeq  (\preceq)~ 0,~   b\in \mathcal{R}(B)~{\rm and}~-b^TB^+b+d=0.\label{12a}
\end{equation}
\end{prop}
%\textcolor{red}{ Assumption \ref{as1} is actually a direct extension
%of the Slater assumption from the inequality case to the equality
%case. Jeyakumar et al. \cite{J09a} showed that the Slater assumption
%does not hold if and only if there is a quadratic function $f$ such
%that S-lemma does not hold. A natural question is what is the
%largest set of the quadratic functions $f$ to make S-lemma fail. It
%will be answered in Section 2 for the more general equality case.
%More precisely, we derive the necessary and sufficient condition for
%the S-lemma with equality to hold when Assumption \ref{as1} is
%violated (i.e., under the condition (\ref{12a})). }
We first discuss the special case that both $f$ and $h$ are
homogeneous (i.e., $a=c=b=d=0$), while the non-homogeneous case will
be proved using the homogeneous result.

Suppose $h$ is homogeneous. Then, (\ref{12a}) is reduced to
$B\succeq (\preceq)~ 0$ and (\ref{h0}) becomes $
h(x)=0~\Longleftrightarrow~x=Zy.$ Therefore,
\[
~~({\rm E}_1) \Longleftrightarrow f(Zy)\ge 0, ~\forall y.
\]
Suppose $f$ is also homogeneous. Then,
\begin{equation}
(homogeneous)~~({\rm E}_1) \Longleftrightarrow Z^TAZ\succeq 0.\label{e1}
\end{equation}
On the other hand, when $f$ and $h$ are quadratic forms,
 \begin{equation}
(homogeneous)~~({\rm E}_2) \Longleftrightarrow (\exists \mu\in \Bbb R)~A+\mu B \succeq 0.\label{e2}
\end{equation}
Therefore, suppose Assumption \ref{as1} is violated and both $f$ and
$h$ are quadratic forms, homogeneous version of (${\rm
E_1}$)$\sim$(${\rm E_2}$) is the same as
({\ref{e1}})$\sim$({\ref{e2}}).

Since ({\ref{e2}}) trivially implies ({\ref{e1}}), it is sufficient
to show that ({\ref{e1}}) implies ({\ref{e2}}). When $B\succ 0$,
$A+\mu B \succ 0$ for any sufficiently large $\mu$, so ({\ref{e2}})
is trivially true. We let $B\succeq 0$ but not definite, i.e.,
$Z\neq 0$. Then there are two possibilities:
\begin{itemize}
\item[(a)] Suppose $Z^TAZ\succ 0$. For $x^TBx = 0$, we have $x=Zy$ for some $y\neq0$
and $x^TAx=y^TZ^TAZy>0$. In other words, the system
    \[
    x^TAx \le 0,~x^TBx = 0,~x\neq 0,
    \]
has no solution. By Finsler's Theorem, ({\ref{e2}}) must be
true.
\item[(b)] Suppose $Z^TAZ\succeq 0$ but not definite.
Anstreicher and Wright  \cite{AW} proved in 2000 that
({\ref{e1}}) is equivalent to ({\ref{e2}})   if and only if
$\mathcal{N}(Z^TAZ)=\mathcal{N}(Z^TA^2Z)$.
\end{itemize}

In summary, if $f, h$ are homogeneous and either $h(x)\ge0$ or
$h(x)\le0$, the S-lemma with equality holds for one of the following
three situations: (i) $B\succ(\prec)~0$; (ii) $B\succeq(\preceq)~0$ and $Z^TAZ\succ
0$; and (iii) $B\succeq(\preceq)~0,~Z^TAZ\succeq 0$ and
$\mathcal{N}(Z^TAZ)=\mathcal{N}(Z^TA^2Z)$. Notice that, in
Introduction, we have seen (${\rm E_1}$)$\not\sim$(${\rm E_2}$) for
the example $f(x_1,x_2)=x_1x_2,~h(x_1,x_2)=-x_1^2$. The reason now is
clear that $B\preceq0,~Z^TAZ=0$,  $\mathcal{N}(Z^TAZ)=\Bbb R$, but
$\mathcal{N}(Z^TA^2Z)=\{0\}$.

For nonhomogeneous $f$ and $h$, we first have the following
proposition.

\begin{prop}\label{prop2}
Suppose Assumption \ref{as1} is violated and $f,~h$ are
nonhomogeneous. Then,
\begin{equation}\label{w}
({\rm E}_1) \Longleftrightarrow W=\widetilde{Z}^T \widetilde{A}\widetilde{Z}
=\left[\begin{matrix}Z^TAZ&Z^Ta-Z^TAB^+b\\
a^TZ-b^TB^{+}AZ&~b^TB^{+}AB^+b-2a^TB^+b+c \end{matrix}\right]\succeq 0,
\end{equation}
where $Z$ is a matrix basis of $\mathcal{N}(B)$, and
\begin{equation}
\widetilde{Z}=\left[\begin{matrix}Z&0\\0&1\end{matrix}\right], ~
\widetilde{A}=\left[\begin{matrix} A & a- AB^+b\\
a^T -b^TB^{+}A^T &~b^TB^{+}AB^+b-2a^TB^+b+c \end{matrix}\right].\label{A-tilde}
\end{equation}
\end{prop}
\begin{proof}
Since $f$ is non-homogeneous and Assumption \ref{as1} fails, we have
\begin{eqnarray*}
\inf_{h(x)=0}~f(x)&=& f(-B^+b+Zy) \\
&=&(-B^+b+Zy)^TA(-B^+b+Zy)+2a^T(-B^+b+Zy)+c\\
 &=&\left[\begin{matrix}y\\
1\end{matrix}\right]^T W
\left[\begin{matrix}y\\
1\end{matrix}\right]\ge0,~~\forall y\in\Bbb R^m
\end{eqnarray*}
where $W$ is defined in (\ref{w}). Then, the matrix $W$ is positive
semi-definite since, for any $\gamma\neq 0$,
\[
\left[\begin{matrix}y\\
\gamma\end{matrix}\right]^T W
\left[\begin{matrix}y\\
\gamma\end{matrix}\right]=\gamma^2
\left[\begin{matrix}{y}/{\gamma}\\
1\end{matrix}\right]^T W
\left[\begin{matrix}{y}/{\gamma}\\
1\end{matrix}\right]\ge 0,~\forall y\in\Bbb R^m;
\]
and also for $\gamma=0$,
\[
\left[\begin{matrix}y\\
0\end{matrix}\right]^T W
\left[\begin{matrix}y\\
0\end{matrix}\right]=\lim_{\gamma\rightarrow 0}
\left[\begin{matrix}y\\
\gamma\end{matrix}\right]^T W
\left[\begin{matrix}y\\
\gamma\end{matrix}\right]\ge 0,~\forall y\in\Bbb R^m.
\]
\end{proof}

Notice that (\ref{w}) is the homogeneous representation for the
nonhomogeneous inequality $f(-B^+b+Zy)\ge0$ by lifting one more
dimension to $\widetilde{A}$. The next result is the main theorem of
this section.
%
%It is worth mentioning that when Assumption \ref{as1} is violated
%$h(x)$ is reduced to a linear variety so that the statement (\ref{eq})
%in (${\rm S_1}$) merges into just one quadratic inequality
%$f(-B^+b+Zy)\ge0$. By lifting one more dimension, the nonhomogeneous
%version of (${\rm S_1}$) has a homogeneous representation by
%$\widetilde{A}$. The condition (\ref{w}) for the nonhomogeneous
%statement (\ref{eq}) to be true is nothing but its homogeneous
%counterpart ({\ref{e1}}).

\begin{thm}\label{new1}
Suppose Assumption \ref{as1} is violated and the same notations as
in Proposition \ref{prop2} are adopted. Then, ${\rm (E_1)}$ $\sim$
${\rm (E_2)}$ if and only if one of the following conditions is
satisfied:
\begin{itemize}
\item[{\rm(a)}]
$\widetilde{Z}^T\widetilde{A} \widetilde{Z}\succ 0$;
 \item[{\rm(b)}]$\widetilde{Z}^T\widetilde{A}\widetilde{Z}\succeq 0$ and
$\mathcal{N}(\widetilde{Z}^T\widetilde{A}\widetilde{Z})=
\mathcal{N}(\widetilde{Z}^T\widetilde{A}^2\widetilde{Z})$.
\end{itemize}
\end{thm}
\begin{proof} Since Assumption \ref{as1} fails, by (\ref{12a}), we have $d=b^TB^+b$.
Then, %$f(x)+\mu h(x)\ge0$ of (${\rm S_2}$) has the following linear
%matrix inequality representation:
\begin{eqnarray*}
(\exists \mu\in \Bbb R)~f(x)+\mu h(x)\ge0,~\forall x\in \Bbb R^n
\Longleftrightarrow \left[\begin{matrix}A&a\\a^T&c \end{matrix}\right]
+\mu\left[\begin{matrix}B&b\\b^T&b^TB^+b \end{matrix}\right]%\label{op}
\succeq 0.
\end{eqnarray*}
Using the invertible matrix $\left[\begin{matrix}I&0\\-b^TB^+&1
\end{matrix}\right]$ to homogenize $h$, we obtain
\begin{eqnarray*}
{\rm (E_2)}&\Longleftrightarrow&(\exists \mu\in \Bbb R)~
\left[\begin{matrix}A&a\\a^T&c \end{matrix}\right]
+\mu\left[\begin{matrix}B&b\\b^T&~b^TB^+b \end{matrix}\right]
\succeq 0\nonumber\\
&\Longleftrightarrow&(\exists \mu\in \Bbb R)~ \left[\begin{matrix}I&0\\-b^TB^+&1 \end{matrix}\right]
\left[\begin{matrix}A&a\\a^T&c \end{matrix}\right] \left[\begin{matrix}I&-B^+b\\0&1
\end{matrix}\right]+\mu\left[\begin{matrix}B&0\\0&0\end{matrix}\right]
\succeq 0 \nonumber\\
&\Longleftrightarrow& (\exists \mu\in \Bbb R)~\widetilde{A}
+\mu\left[\begin{matrix}B&0\\0&0\end{matrix}\right]
\succeq 0%\label{m1}
\end{eqnarray*}
where $\widetilde{A}$ was defined in (\ref{A-tilde}). Let
$\widetilde{B}=\left[\begin{matrix}B&0\\0&0\end{matrix}\right]\succeq0.$
Notice that it is not definite and
$\widetilde{Z}=\left[\begin{matrix}Z&0\\0&1\end{matrix}\right]$ the
basis matrix for the null space of $\widetilde{B}$. By applying the
homogeneous version (${\ref{e1}}$)$\sim$(${\ref{e2}}$) to
$\widetilde{A}$ and $\widetilde{B}$, the conclusion of the theorem
follows immediately.
\end{proof}

Recall that, the classical S-lemma ${\rm (S_1)}$ $\sim$ ${\rm
(S_2)}$ relies on Slater's condition. When the condition is absent,
the regularized S-lemma ${\rm (S_1)}$ $\sim$ ${\rm (S^r_2)}$ is a
substitute for ${\rm (S_1)}$ $\sim$ ${\rm (S_2)}$ (\cite{J09a}).
Following the similar idea and as a direct consequence of Theorem
\ref{new1}, we can also formulate the following regularized version
of S-lemma with equality in the absence of Assumption \ref{as1}:
 \begin{thm}\label{new1:2}
Suppose Assumption \ref{as1} is violated. Then, the following two
statements are equivalent:
\begin{itemize}
\item[]${\rm (E_1)}$~~$(\forall x\in\Bbb R^n)$ $~~h(x)=0 ~\Longrightarrow~ f (x) \ge 0.$
\item[]${\rm (E^r_2)}$~~ $(\forall \epsilon>0) (\exists \lambda_{\epsilon})
(\forall x\in\Bbb R^n) ~f(x)+ \lambda_{\epsilon}h(x)+ \epsilon(x^Tx + 1) \ge 0$.
\end{itemize}
\end{thm}
\begin{proof}
Since ${\rm (E^r_2)}\Longrightarrow{\rm (E_1)}$ is trivial by
letting $\epsilon\rightarrow 0$, it is sufficient to prove the
converse. To this end, for all $\epsilon>0$, consider
$f_{\epsilon}(x)=f(x)+\epsilon(x^Tx + 1)=x^T(A+\epsilon
I)x+2a^Tx+(c+\epsilon)$ and $h(x)$. By ${\rm (E_1)}$, we first have
\begin{equation}
(\forall \epsilon>0) (\forall x\in\Bbb R^n)~~h(x)=0 ~\Longrightarrow~ f(x)\ge0
~\Longrightarrow~~f_\epsilon (x)\ge 0.\label{f_epsilon}
\end{equation}
Secondly, from Proposition \ref{prop2}, (E$_1$) also implies that
$W\succeq0$ in (\ref{w}). Then, due to $Z^TZ=I$ and $Z^TB^+=0$,
there is
\begin{eqnarray*}
&&\left[\begin{matrix}Z^T\left(A+\epsilon I\right)Z&Z^Ta-Z^T\left(A+\epsilon I\right)B^+b\\
a^TZ-b^TB^{+}\left(A+\epsilon I\right)Z&\ b^TB^{+}(A+\epsilon I)B^+b-2a^TB^+b+(c+\epsilon)
\end{matrix}\right] \\
&=&
\left[\begin{matrix}Z^TAZ&Z^Ta-Z^T AB^+b\\
a^TZ-b^TB^{+}AZ&\ b^TB^{+}AB^+b-2a^TB^+b+c\end{matrix}\right]+\epsilon
\left[\begin{matrix}I&~0\\
0&~\|B^+b\|^2+1\end{matrix}\right]
\succ 0.
\end{eqnarray*}
In other words, Case (a) in Theorem \ref{new1} holds for
$f_{\epsilon}(x)$ and $h(x)$. The S-lemma with equality for the pair
$f_{\epsilon}(x)$ and $h(x)$ thus implies that, from
(\ref{f_epsilon}), there must be an associated
$\lambda_{\epsilon}\in\Bbb R$ for any $\epsilon>0$ such that
\[
f_{\epsilon}(x)+\lambda_{\epsilon}h(x)
=f(x)+ \lambda_{\epsilon}h(x) + \epsilon(x^Tx + 1)\ge 0,~\forall x\in\Bbb R^n,
\]
which proves the theorem.
\end{proof}

Theorem \ref{new1:2} generalizes the regularized S-lemma. Consider
$\widehat{h}(x,z)=h(x)+z^2=0.$ Assume that $h(x)\le0$ fails Slater's
condition. Then, $\widehat{h}(x,z)=h(x)+z^2\ge0, \forall x\in\Bbb
R^n,~\forall z\in\Bbb R$ fails Assumption \ref{as1}. Since ${\rm
(S_1)}$ can be equivalently rephrased as
\begin{equation}
(\forall x\in\Bbb R^n,~\forall z\in\Bbb R)~~\widehat{h}(x,z)=h(x)+z^2=0~\Longrightarrow~
\widehat{f}(x,z)=f(x)\ge 0,\label{S1-eqiv}
\end{equation}
we can apply Theorem \ref{new1:2} to $(\widehat{f},\widehat{h})$ and
obtain
\begin{equation}
(\forall \epsilon>0) (\exists \lambda_{\epsilon})
(\forall x\in\Bbb R^n,~\forall z\in\Bbb R)
~f(x)+ \lambda_{\epsilon}(h(x)+z^2)+ \epsilon(x^Tx + 1) \ge 0. \label{lm-2}
\end{equation}
Let $z\rightarrow \infty$ in (\ref{lm-2}), we have
$\lambda_{\epsilon}\ge 0$ for all $\epsilon>0$. The validity of
${\rm (S^r_2)}$, and hence ${\rm (S_1)}$ $\sim$ ${\rm (S^r_2)}$, is
concluded from setting $z=0$ in (\ref{lm-2}).

The original proof of the regularized S-lemma in \cite{J09a}, based
on Brickman's theorem \cite{Br} and Mart\'{\i}nez-Legaz's result
\cite{Mar}, is a bit tedious. Our argument by applying Theorem
\ref{new1:2} can be viewed as a direct consequence of Finsler's
Theorem \cite{F}, which is more direct and simple.

%
%Now, we can reprove the regularized S-lemma using a simple proof
%similar to that of Theorem \ref{new1:2}, based on Finsler's Theorem
%\cite{F}, which plays a key role in proving Case (a) in Theorem
%\ref{new1}.

\section{The S-lemma with equality when Assumption \ref{as1} holds}

Throughout this section, we always assume that $h$ takes both
positive and negative values (Assumption \ref{as1}). It is frequent
to consider the homogenized version by introducing a new variable
$t\in \Bbb R$ as follows:
\begin{eqnarray*}
&&\widetilde{f}(x,t)=x^TAx+2ta^Tx+ct^2,  \\
&&\widetilde{h}(x,t)=x^TBx+2tb^Tx+dt^2.
\end{eqnarray*}
If $t\not=0$, (${\rm E_2}$) implies that
\[
(\exists\mu\in\mathbb{R})(\forall x\in \Bbb R^n)~\widetilde{f}(x,t) + \mu \widetilde{h}(x,t)=
t^2\left(f\left(\frac{x}{t}\right) + \mu h\left(\frac{x}{t}\right)\right)
\ge0.
\]
For $t=0$, there is
\[
(\exists\mu\in\mathbb{R})(\forall x\in \Bbb R^n)~\widetilde{f}(x,0) + \mu \widetilde{h}(x,0)=\lim_{t\rightarrow 0}
\widetilde{f}(x,t) + \mu \widetilde{h}(x,t) \ge0.
\]
Consequently, the validity of (${\rm E_2}$) implies that of its
homogenized version
\begin{itemize}
\item[]($\widetilde{\rm E}_2$)~~$(\exists\mu\in\mathbb{R})(\forall x\in
\Bbb R^n, \forall t\in \mathbb{R})$ $\widetilde{f}(x,t) + \mu
\widetilde{h}(x,t)\ge0,$
\end{itemize}
and vice versa (by setting $t=1$). So we have (${\rm
E_2}$)$\sim$($\widetilde{\rm E}_2$). On the other hand, by the
S-lemma with equality for a homogeneous quadratic system under
Assumption \ref{as1} and S-Condition 3, ($\widetilde{\rm
E}_2$)$\sim$($\widetilde{\rm E}_1$) below:
\begin{itemize}
\item[]($\widetilde{\rm E}_1$)~$(\forall x\in
\Bbb R^n, \forall t\in \mathbb{R})$
$~x^TBx+2tb^Tx+dt^2=0~\Rightarrow~x^TAx+2ta^Tx+ct^2\ge 0.$
\end{itemize}
Comparing ($\widetilde{\rm E}_1$) with (${\rm E_1}$), we only know
they are equivalent when $t\not=0$, by rewriting
\begin{equation}
(\forall x\in \mathbb{R}^n,t\not=0)~x^TBx+2 tb^Tx+dt^2 =0
\Longrightarrow x^TAx+2 ta^Tx+ct^2 \ge 0
\label{eq:new1}
\end{equation}
as
\[(\forall x\in \mathbb{R}^n)~
\left(\frac{x}{t}\right)^TB\left(\frac{x}{t}\right)+2 b^T
\left(\frac{x}{t}\right)+d =0 \Longrightarrow
\left(\frac{x}{t}\right)^TA\left(\frac{x}{t}\right)+2
a^T\left(\frac{x}{t}\right)+c \ge 0.
\]
Therefore, if we are to argue that (${\rm E_1}$)$\sim$(${\rm E_2}$),
it amounts to finding conditions under which (${\rm E_1}$)
$\Longrightarrow$ ($\widetilde{\rm E}_1$) for $t=0$. That is, we
need conditions for the following compound statement to hold:
\begin{equation}
\left\{({\rm
E_1}): \inf\limits_{h(x)=0}f(x)\ge0\right\} \Longrightarrow
\left\{(\forall x\in \mathbb{R}^n)~x^TBx =0~\Longrightarrow~x^TAx
\ge 0\right\}. \label{eq:new2}
\end{equation}

%We first analyze the right hand side of (\ref{eq:new2}).
%\begin{prop}\label{as6}
%The statement
%\begin{equation}
%x^TBx =0~\Longrightarrow~x^TAx
%\ge 0,~\forall x\in \mathbb{R}^n \label{eq:new}
%\end{equation}
%is true if and only if exactly one of the following three cases
%happens:
%\begin{itemize}
%\item[(a)]  $B\succ 0$ or $B\prec0$.
%\item[(b)]  $B\succeq 0$ or  $B\preceq 0$ but not definite, and
%$Z^TAZ\succeq 0$ ( where $Z$ again is the matrix basis of
%$\mathcal{N}(B)$).
%\item[(c)]  $B$ is indefinite, and there is an $\eta\in \Bbb R$ such that
%$ A \succeq\eta B.$
%\end{itemize}
%\end{prop}
%\begin{proof}
%We first notice that the three cases (a), (b) and (c) are exclusive.
%When $B\succ 0$ or $B\prec0$, (\ref{eq:new}) is clearly true. When
%$B\succeq 0$ or $B\preceq 0$ but not definite, (\ref{eq:new}) is
%just (${\rm E_1}$) with $f(x)=x^TAx;~h(x)=x^T Bx$ and Assumption
%\ref{as1} fails. The case belongs to the discussion in Sect. 2, and
%(\ref{eq:new}) is thus equivalent to $Z^TAZ\succeq0$ by (\ref{e1}).
%When $B$ is indefinite, the homogeneous function $h(x)=x^T Bx$ must
%take both positive and negative values. Apply S-Condition 3 to
%establish the equivalence between (\ref{eq:new}) and $(\exists
%\mu\in\Bbb R)~x^TAx+\mu x^TBx\ge0,~\forall x\in \mathbb{R}^n$. The
%proof is completed by setting $\eta=-\mu$.
%\end{proof}

We now present and prove the necessary and sufficient conditions for
the S-lemma with equality to be valid. %We show that, under $B=0$,
%(${\rm E_1}$) is true but
%\begin{equation}
%x^TBx =0~\Longrightarrow~x^TAx
%\ge 0,~\forall x\in \mathbb{R}^n \label{eq:new}
%\end{equation}
%fails (in which case (${\rm
%E_1}$)$\not\sim$(${\rm E_2}$)) if, and only if $A$ has exactly one
%negative eigenvalue; $b\not=0$; and the matrix in (\ref{m}) is
%positive semi-definite. Conversely, if $B\not=0$ and (\ref{eq:new})
%fails, then we can always prove that
%$\inf\limits_{h(x)=0}f(x)=-\infty$ and therefore (${\rm E_1}$) is
%invalid. In other words, when $B\not=0$, the S-lemma with equality
%must be correct. It fails only when $B=0$; $b\not=0$; $A$ has
%exactly one negative eigenvalue; and the matrix in (\ref{m}) is
%positive semi-definite.
%
%We have the following main theorem:

\begin{thm}\label{thm:m}
Under Assumption \ref{as1} that $h(x)$ takes both positive and
negative values, the S-lemma with equality holds except that $A$ has
exactly one negative eigenvalue, $B=0$, $b\neq 0$ and
\begin{equation}
\left[\begin{array}{cc}V^TAV&V^T(Ax_0+a)\\
(x_0^TA+a^T)V&f(x_0)\end{array}\right]\succeq 0,\label{m}
\end{equation}
where $x_0=-\frac{d}{2b^Tb}b$, $V\in \Bbb R^{n\times (n-1)}$ is the
matrix basis of $\mathcal{N}(b)$.
\end{thm}
\begin{proof} The statement of the theorem is the same as
(${\rm E_1}$)$\not\sim$(${\rm E_2}$) if, and only if $B=0,~b\not=0$,
$A$ has exactly one negative eigenvalue and $(\ref{m})$ holds. The
proof is organized in the following sequence:

(Sufficiency) Suppose $B=0,~b\not=0$, $A$ has exactly one negative
eigenvalue and $(\ref{m})$ holds. We show that (\ref{eq:new2}) is a
false statement. That is, (${\rm E_1}$)$\not\sim$(${\rm E_2}$).

(Necessity) Suppose (\ref{eq:new2}) is a false statement such that
(${\rm E_1}$) is true but
\begin{equation}
(\forall x\in \mathbb{R}^n) ~~x^TBx =0~\Longrightarrow~x^TAx
\ge 0\label{eq:new}
\end{equation}
fails. Through a case-by-case analysis (cases (a), (b), (c-1), (c-2)
and (c-3) below), we show that there must be $B=0$. Once this is
obtained, it follows immediately that $b\not=0$, $(\ref{m})$ holds,
and $A$ has exactly one negative eigenvalue.\vskip0.2cm

\noindent{\bf(Proof for sufficiency)} From $B=0,~b\not=0$, we
observe that the set $\{x\in\Bbb R^n:h(x)=0\}$ is a linear variety
of $n-1$ dimension
\begin{equation}
\{x\in\Bbb R^n:h(x)=0\}=\{x_0+Vy: y\in \Bbb R^{n-1}\}\label{eq:new:1}
\end{equation}
with $x_0=-\frac{d}{2b^Tb}b$ being a particular solution of $h(x)=0$
and $V$ is a matrix basis of $\mathcal{N}(b)$. Then, (${\rm E_1}$)
becomes an unconstrained minimization problem
\begin{equation}
\inf_{y\in \Bbb R^{n-1}}\left\{ f(x_0+Vy)=f(x_0)+2(x_0^TA+a^T)Vy+y^TV^TAVy\right\}\ge 0\label{eq:new:2}
\end{equation}
where ``$\ge0$'' comes from $(\ref{m})$. That is, (${\rm E_1}$) is
true. However, from $B=0$ and $A\not\succeq0$, we know
(\ref{eq:new}) is wrong so that (${\rm E_1}$)$\not\sim$(${\rm E_2}$)
and the sufficiency is proved.\vskip0.2cm

\noindent {\bf(Proof for necessity)} We defer the proof for $B=0$ to
later discussions but first assume that it is true. We show that,
when $B=0$, (${\rm E_1}$) is true and (\ref{eq:new}) fails, it
follows immediately that $b\not=0$, $(\ref{m})$ holds, and $A$ has
exactly one negative eigenvalue.

Since $h(x)$ takes both positive and negative
values, $B=0$ implies that $b\not=0$ and then the set $\{x\in\Bbb
R^n:h(x)=0\}$ is the $n-1$ dimensional linear variety
(\ref{eq:new:1}). Since (${\rm E_1}$) is true, we obtain
(\ref{eq:new:2}), which gives $(\ref{m})$ and $V^TAV\succeq 0$. As
$B=0$ and (\ref{eq:new}) fails, we have $A\not\succeq0$. Then, the
matrix $A$ must have exactly one negative eigenvalue.

In the remaining part of the necessity proof, we show $B=0$ by
contradiction. That is, we are going to argue that ``if $B\not=0$
and (\ref{eq:new}) fails, then $\inf\limits_{h(x)=0}f(x)=-\infty$
(so (${\rm E_1}$) fails too).''

Since (\ref{eq:new}) fails, there is a $v\in \Bbb R^n$ such that
\begin{equation}
v^TBv=0,~v^TAv<0. \label{v}
\end{equation}
This $v$ is to be utilized to construct a curve
$\{y+\alpha(y)v:~y\in\Bbb R^n,~\alpha(y)\in \Bbb R\}$ on
$\{x:~h(x)=0\}$ upon which  $f$ is unbounded below.

For this purpose, we derive the following formulae.
\begin{eqnarray}
h(y+\alpha(y)v)&=&(y+\alpha(y)v)^TB(y+\alpha(y)v)+2b^T(y+\alpha(y)v)+d\nonumber\\
&=&h(y)+2y^T(Bv)\alpha(y)+2(b^Tv)\alpha(y);\label{curve-1}
\end{eqnarray}
and
\begin{eqnarray}
f(y+\alpha(y)v)&=&(y+\alpha(y)v)^TA(y+\alpha(y)v)+2a^T(y+\alpha(y)v)+c\nonumber\\
&=&f(y)+2y^T(Av)\alpha(y)+2(a^Tv)\alpha(y)+(v^TAv)(\alpha(y))^2.\label{curve-2}
\end{eqnarray}

From (\ref{curve-1}), we can construct $\{y+\alpha(y)v:~y\in\Bbb
R^n,~\alpha(y)\in \Bbb R\}$ on $\{x:~h(x)=0\}$ in different ways
based on three cases:
 Case (a) for $b^Tv=0$; Case (b) for
$b^Tv\neq0,~Bv=0$ and Case (c) for $b^Tv\neq0,~Bv\neq0$. For
convenience, we may assume $d=0$, i.e., $h(0)=0.$ If this is not the
case, we choose a nonzero vector $x'\in \Bbb R^n$ such that
$h(x')=0$ and then translate the origin there.

\begin{itemize}
\item[(a)] Suppose $b^Tv=0$. In this case, we choose $y=0$ and let $\alpha(y)$
be any real number $\alpha$ in (\ref{curve-1}). Due to $h(0)=0$,
it is easy to see that $h(\alpha v)=0$ for all $\alpha$.
Therefore, the set $\{x:~h(x)=0\}$ contains a straight line
$\{\alpha v:~\alpha\in\Bbb R\}$ which goes through the origin in
the direction of $v$. The values of $f$ on this line can be
easily read from (\ref{curve-2}) to have
\[
\inf_{\alpha}\left\{f(\alpha v)= c+2(a^Tv)\alpha+(v^TAv)\alpha^2\right\}=-\infty
\]
where $v^TAv<0$ due to (\ref{v}). Therefore, we have
$\inf\limits_{h(x)=0}f(x)=-\infty$.
\item[(b)] Suppose $Bv=0$ and $b^Tv\neq0$. For any $y\in \Bbb R^n$, let
\[
\alpha(y)=-\frac{h(y)}{2b^Tv}.
\]
Then, $y+\alpha(y) v$ satisfies $h(y+\alpha(y)
v)=h(y)+2(b^Tv)\alpha(y)=0.$ Since $B\not=0$, there is an
eigenvector $u\not=0$ corresponding to a nonzero eigenvalue
$\sigma$ of $B$, i.e., $Bu=\sigma u$. Consider $y$ to be the
line $\{\gamma u:~\gamma\in\Bbb R\}$ spanned by $u$. Then,
$\Gamma=\{\gamma u+\alpha(\gamma u) v:~\gamma\in\Bbb R\}$ is a
curve on $\{x:~h(x)=0\}$. The values of $f$ on $\Gamma$ can be
verified to be unbounded below by
\begin{eqnarray*}
&&\inf_{\gamma\in\Bbb R}\left\{f(\gamma u+\alpha(\gamma u) v)\right\} \\
&=&\inf_{\gamma\in\Bbb R}\left\{f(\gamma u)-2(\gamma Au+a)^T v
\frac{h(\gamma u)}{2b^T v} +v^T A v
\frac{h^2(\gamma u)}{(2b^T v)^2}\right\}\\
&=&-\infty%\\
%=&&\inf_{\gamma\in \Bbb R}~\left\{f(\gamma u)-2(\gamma Au+a)^T v
%\frac{(\sigma\gamma^2\|u\|^2+2\gamma b^Tu)}{2b^T v}
%+v^T A v\frac{(\sigma\gamma^2\|u\|^2+2\gamma b^Tu)^2}{(2b^T v)^2}\right\}\\
%=&&-\infty,
\end{eqnarray*}
since the coefficient $\frac{v^T A v}{(2b^T v)^2}<0$ and
$h^2(\gamma u)=(\sigma\gamma^2\|u\|^2+2\gamma b^Tu)^2$ is a
polynomial of degree 4 in $\gamma$ whereas $f(\gamma u)$ is only
of degree 2 and $\gamma h(\gamma u)$ is of degree 3.

\item[(c)] Suppose $Bv\neq 0$ and $b^Tv\neq0$.
In this case, $B$ is indefinite; otherwise
$B\succeq(\preceq)~0$, $v^TBv=0$ would imply that $Bv=0$.
Without loss of generality, assume that $B$ is diagonal after
performing the eigenvalue decomposition on $B$. We also assume
that  $B={\rm Diag}( B_{ii})$ where
\[
B_{ii}=\left\{\begin{array}{ccc}1,& i \in I,\\-1,& i \in J,\\
0,&i \in \{1,2,\ldots,n\}\setminus(I\cup J).\end{array}\right.
\]
It will soon be clear that only the signs of the entries matter.
Since $B$ is indefinite, both $I$ and $J$ are non-empty, i.e.,
$\#I \ge 1$ and $\#J \ge 1$. It follows that the rank of $B$ is
at least 2. When Rank$(B)=2$, the homogeneous quadratic surface
$x^TBx=0$ is %degenerate into
the union of two vertical-like hyperplanes (a type of cylindroid
in geometry), which will be dealt with separately. When
Rank$(B)=3$, $x^TBx=0$ is a second order cone (a double circular
cone). When Rank$(B)>3$, it is sure that $x^TBx=0$ is not the
union of hyperplanes since there is at least one
three-dimensional second order cone embedded as a cross section.
\begin{itemize}
\item[(c-1)] Suppose Rank$(B)\ge 3.$ We assume that $I=\{i_1,i_2,\ldots,i_m\}$
and $J=\{j_1,j_2,\ldots,j_k\}$ with $m\ge2,~k\ge1$. We first
claim that it is always possible to choose one $v$ such that
$v^TBv=0,~v^TAv<0$ and $v_{i}^2\neq v_{j}^2$ for some $i\in
I$ and $j\in J$. If the vector $v$ satisfying (\ref{v}) does
not meet the requirement, namely $v_{i}^2= v_{j}^2$ for all
$i\in I$ and $j\in J$, we can perturb $v$ by $\epsilon w$
where $\epsilon>0$ is a sufficiently small constant; and
\[
w_{i}=\left\{\begin{array}{ccc}1,& i\in\{i_2,j_1\}\\0,& o.w.\end{array}\right.
\]
Then, $Bw=0$; $(v+\epsilon w)^TB(v+\epsilon w)=0$;
$(v+\epsilon w)^TA(v+\epsilon w)<0$; and $(v_{i_1}+\epsilon
w_{i_1})^2=v^2_{i_1}\not=(v_{j_1}+\epsilon)^2=(v_{j_1}+\epsilon
w_{j_1})^2$. So we can assume that $v^TBv=0,~v^TAv<0,$
$Bv\neq 0$, $b^Tv\neq0$ and $v_{i_1}^2\neq v_{j_1}^2$.

Define a straight line
\begin{equation}
x_{\beta}=\beta (v_{j_1}) e_1+\beta(v_{i_1})
e_2,~\forall \beta\in \Bbb R.\label{x_beta}
\end{equation}
Since $ x_{\beta}^TBv=0,~\forall \beta\in \Bbb R,$ the line
$x_{\beta}$ and the vector $v$ are conjugate with respect to
$B$. By defining
\[
\alpha(x_{\beta})=-\frac{h(x_{\beta})}{2b^Tv},
\]
and by the conjugacy, we see from (\ref{curve-1}) that
\[
h(x_{\beta}+\alpha(x_{\beta}) v)=h(x_{\beta})+2(b^Tv)\alpha(x_{\beta})=0
\]
and $x_\beta+\alpha v$ is a curve on $h(x)=0$ upon which
$f(x)$ is unbounded below as
\begin{eqnarray*}
&&\inf_{h(x_{\beta}+\alpha v)=0}~\left\{f(x_{\beta}+\alpha v)=f(x_{\beta})
+2(Ax_{\beta}+a)^T v \alpha +v^T A v\alpha^2\right\}\\
&&~~=\inf\limits_{\beta\in \mathbb{R}}~\left\{f(x_{\beta})-2(Ax_{\beta}+a)^T v
\frac{h(x_{\beta})}{2b^Tv}
+v^T Av \frac{h^2(x_{\beta})}{(2b^Tv)^2}\right\}\\
&&~~=-\infty,
\end{eqnarray*}
where the coefficient $\frac{v^T A v}{(2b^T v)^2}<0$ and due
to $(v_{j_1})^2 -(v_{i_1})^2\not=0$,
$$h^2(x_{\beta})=\left(\beta^2 (v_{j_1})^2 - \beta^2
(v_{i_1})^2 +2\beta b_{i_1} v_{j_1} + 2\beta b_{j_1}
v_{i_1}\right)^2$$ is a polynomial of degree 4 in $\beta$
whereas $f(x_{\beta})$ is only of degree 2 and $\beta
h(x_{\beta})$ is of degree 3. Finally, we remark that when
the diagonal elements of $B$ are not just $0, 1, -1$, the
line $x_{\beta}$ defined in (\ref{x_beta}) can be adjusted
through the linear combination of $e_1$ and $e_2$ to
maintain the conjugacy to $v$ and the rest of the proof
follows immediately.

\item[(c-2)] Suppose Rank$(B)= 2$ with $B_{11}=1$, $B_{22}=-1$ and $B_{ii}=0$,
for $i\ge3$. In this subcase (c-2), we handle
$\widetilde{b}=(b_3,\ldots,b_n)^T\neq 0$ whereas leaving
$\widetilde{b}=0$ to (c-3) next. Since $x^TBx=0$ is the
union of two cylindroid hyperplanes and
$\widetilde{b}\not=0$, we show that there is an oblique
cross section of $h(x)=0$, which contains a straight line in
the direction of $v$ for any $v$ satisfying (\ref{v}) and
$Bv\neq 0$. Since it must be $v_1^2=v_2^2\neq 0$, the line
has a formula
\begin{equation*}
l(t)=\left(\begin{array}{c}0\\y\\ \widetilde{z}\end{array}\right)
+tv,~t\in\Bbb R~~\hbox{ and }~~ h(l(t))=0
\end{equation*}
where
$$y=\frac{b^Tv}{v_2};~\widetilde{z}=\frac{(\frac{{b^Tv}}{v_2})^2-\frac{2b_2(b^Tv)}{v_2}}
{2\widetilde{b}^T\widetilde{b}} \widetilde{b}.$$ Denoting
$x_0^T=(0,y,\widetilde{z}^T)$ and substituting $l(t)$ in
$f(x)$ yields
$$f(x_0+tv)=(v^TAv)t^2+2(x_0^TAv+a^Tv)t+f(x_0),$$
which tends to $-\infty$ as $|t|\longrightarrow \infty$ due
to $v^TAv<0.$
\item[(c-3)] Suppose that Rank$(B)= 2$ with $B_{11}=1$, $B_{22}=-1$, $B_{ii}=0$,
for $i\ge3$ and $\widetilde{b}= 0$. Then, the optimization
problem $\inf\limits_{h(x)=0} f(x)$ is unconstrained in the
last $n-2$ variables. Denote
\[
A=\left [ \begin{array}{cc}A_1&A_2\\A_2^T& A_3\end{array}\right]
\]
where $A_1\in\Bbb R^{2\times 2}$. Let $A_3=U^T\Sigma U$ be
the eigenvalue decomposition with $U$ orthogonal and
$\Sigma={\rm Diag}(\sigma_i)$. The Moore-Penrose generalized
inverse of $A_3$ is $A_3^+=U^T\Sigma^+ U$, where
$\Sigma^+={\rm Diag}(\sigma_i^{-1})$ when $\sigma_i\not=0$
and $0^{-1}=0$. Define
\[
W=\left [ \begin{array}{cc}I_2&-A_2A_3^+\\0& U\end{array}\right].
\]
Then we have
\[
W
\left [ \begin{array}{cc}A_1&A_2\\A_2^T& A_3\end{array}\right]
W^T
=\left [ \begin{array}{cc}\widehat{A}_1&0\\0& \Sigma\end{array}\right],~
WBW^T
=\left [ \begin{array}{ccc}1&0&0\\0&-1& 0\\0&0&0_{n-2}\end{array}\right]=B
\]
where $\widehat{A}_1=A_1-A_2A_3^+A_2^T$. Introducing the
coordinate change
\begin{equation*}
\left(\begin{array}{c}z\\y\end{array}\right)=W^{-T}x;~\left(\begin{array}{c}\widehat{a}_z\\\widehat{a}_y
\end{array}\right)=Wa;~\widehat{b}=\left(\begin{array}{c}\widehat{b}_1\\\widehat{b}_2
\end{array}\right)=Wb
\end{equation*} where $z\in\Bbb
R^2,~y\in\Bbb R^{n-2}$ yields
\begin{eqnarray}
\inf_{h(x)=0}~f(x) = &\inf& z^T\widehat{A}_1z+2\widehat{a}_z^Tz+y^T\Sigma y
+2\widehat{a}_y^Ty+c \label{y-convex}\\
&{\rm s.t.}& z^T\left[\begin{array}{cc}1&0\\0&-1\end{array}\right]z+2\widehat{b}^Tz=0,
~z\in \Bbb R^2,~y\in \Bbb R^{n-2}.\nonumber
\end{eqnarray}
Obviously, due to $B_{ii}=0,~i\ge3$ and
$\widetilde{b}=(b_3,b_4,\ldots,b_n)= 0$, the variable $y$ is
unconstrained. If there is some $\sigma_i<0$ (in which case
$A_3\not\succeq0$); or some $\sigma_i=0$ but
$(\widehat{a}_y)_i\not=0$ (in which case at least one of the
two columns of $A_2^T$ is not in the range of $A_3$), the
problem $\inf_{h(x)=0}~f(x)$ is surely $-\infty$. Therefore,
we only have to concentrate on the case that, for all
$i\in\{3,4,\ldots,n\}$, either $\sigma_i>0$ or
$\sigma_i=(\widehat{a}_y)_i=0$. That is, the function
$$y^T\Sigma y+2\widehat{a}_y^Ty+c$$
is a convex sum-of-squares quadratic with no pure linear
terms, which has a global optimal solution
$y^*=-\Sigma^+\widetilde{a}_y$. Substituting $y^*$ into
(\ref{y-convex}), it reduces to the following
quadratic optimization problem with two variables:
\begin{eqnarray}
\inf_{h(x)=0}~f(x) = &\inf\limits_{z\in \Bbb R^2}& z^T\widehat{A}_1z+2\widehat{a}_z^Tz-
\widehat{a}_y^T\Sigma^+\widehat{a}_y+c \label{p1}\\
&{\rm s.t.}& z^T\left[\begin{array}{cc}1&0\\0&-1\end{array}\right]z+2\widehat{b}^Tz=0.\label{p2}
\end{eqnarray}
Notice that the equation (\ref{p2}) can be used to solve
either $z_1$ or $z_2$. To see this, we first suppose that
$|\widehat{b}_1|\ge |\widehat{b}_2|$. By introducing
$$\delta=\sqrt{\widehat{b}_1^2-\widehat{b}_2^2}-\widehat{b}_1,~
\widehat{z}_1=z_1-\delta,~
\widehat{z}_2=z_2-\widehat{b}_2,$$ we obtain a new expression for (\ref{p2}):
\begin{equation}
    z_1^2-z_2^2+2\widehat{b}_1 z_1+2\widehat{b}_2 z_2
    =\widehat{z}_1^2-\widehat{z}_2^2+2(\widehat{b}_1+\delta)\widehat{z}_1 \label{elim-l-1}
\end{equation}
in which there is no linear term of $\widehat{z}_2$.
Similarly for $|\widehat{b}_1|< |\widehat{b}_2|$, we can
eliminate the linear term of $\widehat{z}_1$ by letting
\begin{equation}
\delta=\sqrt{\widehat{b}_2^2-\widehat{b}_1^2}+\widehat{b}_2,~
\widehat{z}_1=z_1+\widehat{b}_1,~
\widehat{z}_2=z_2-\delta. \label{elim-l-2}
\end{equation}
We also notice that the transformation (\ref{elim-l-1}) or
(\ref{elim-l-2}) only affect the linear term $\widehat{a}_z$
of (\ref{p1}) but not the second order term $\widehat{A}_1$.
Consequently, we may assume $|\widehat{b}_1|\ge
|\widehat{b}_2|=0$ so that (\ref{p2}) can be solved as
\begin{equation}
z_2=\pm\sqrt{z_1^2+2\widehat{b}_1z_1}.\label{z_2}
\end{equation}
Denote $\widehat{A}_1=\left[\begin{array}{cc}
\widehat{a}_{11}&\widehat{a}_{12}\\\widehat{a}_{12}&
\widehat{a}_{22}\end{array}\right]$ and substitute
(\ref{z_2}) into (\ref{p1}). It becomes
\begin{eqnarray*}
\inf_{h(x)=0}~f(x) &\le&\inf\limits_{|z_1|\gg1}\left\{\widehat{a}_{11}z_1^2\pm 2
\widehat{a}_{12}
z_1\sqrt{z_1^2+2\widehat{b}_1z_1}+ \widehat{a}_{22}(z_1^2+2\widehat{b}_1z_1)\right. \\
&&~~ \left. +2(\widehat{a}_z)_1 z_1\pm2(\widehat{a}_z)_2  \sqrt{z_1^2+2\widehat{b}_1z_1}
- \widehat{a}_y^T\Sigma^+\widehat{a}_y+c\right\}\\
&= &\inf_{|z_1|\gg1} \left\{\left(\widehat{a}_{11}- 2 |\widehat{a}_{12}|
+ \widehat{a}_{22}\right)z_1^2+O(z_1)\right\}
%&=&-\infty,
\end{eqnarray*}
It remains to show, due to $v^TBv=0,~v^TAv<0,~Bv\not=0$ and
that we are in the case $\Sigma\succeq0$, the leading term
$\widehat{a}_{11}- 2 |\widehat{a}_{12}| +
\widehat{a}_{22}<0.$ We first perform the coordinate change
$\widehat{v}=W^{-T}v$ to give
 \begin{eqnarray}
 v^TBv=0&\Longrightarrow&\widehat{v}^T\left [ \begin{array}{ccc}1&0&0\\0&-1& 0\\0&0&0_{n-2}\end{array}
 \right]\widehat{v}=0~\Longrightarrow~\widehat{v}_1^2=\widehat{v}_2^2;\label{hat_v}\\
v^TAv<0&\Longrightarrow&\widehat{v}^T\left [ \begin{array}{cc}\widehat{A}_1&0\\0&
\Sigma\end{array}
\right]\widehat{v}<0
~\Longrightarrow~ \left[\begin{array}{c}\widehat{v}_1\\ \widehat{v}_2\end{array}\right]^T
\widetilde{A}_1
\left[\begin{array}{c}\widehat{v}_1\\ \widehat{v}_2\end{array}\right]<0.\label{hat_a}
\end{eqnarray}
Moveover, $Bv\neq 0$ implies that
$BW^T\widehat{v}=B\widehat{v}\neq 0$, i.e., either
$\widehat{v}_1\neq0$ or $\widehat{v}_2\neq0$. From
(\ref{hat_v}), it must be
$\widehat{v}_1=\widehat{v}_2\not=0$ or
$\widehat{v}_1=-\widehat{v}_2\not=0$. It follows from
(\ref{hat_a}) that one of the following equations holds:
\begin{eqnarray*}
&&\widehat{a}_{11}+\widehat{a}_{22}+2\widehat{a}_{12}<0,~\hbox{ for }
\widehat{v}_1=\widehat{v}_2;\\
&&\widehat{a}_{11}+\widehat{a}_{22}-2\widehat{a}_{12}<0, ~\hbox{ for }
\widehat{v}_1=-\widehat{v}_2.
\end{eqnarray*}
Combining together, we have
$\widehat{a}_{11}+\widehat{a}_{22}<2|\widehat{a}_{12}|.$ The
final subcase and the entire proof of the theorem is thus
complete.
\end{itemize}
\end{itemize}

\end{proof}

Theorem \ref{thm:m} generalizes the classical S-lemma (${\rm
S_1}$)$\sim$(${\rm S_2}$) under Slater's condition. First, like
(\ref{S1-eqiv}), we rewrite (${\rm S_1}$) in terms of
$\widehat{h}(x,z)=h(x)+z^2$ and $\widehat{f}(x,z)=f(x)$. Let
$\overline{x}$ be a Slater point of $h(x)\le0$. Then, Assumption
\ref{as1} can be easily checked to hold by
\[
\widehat{h}(\overline{x},0)=h(\overline{x})<0,~
\widehat{h}(\overline{x},1-h(\overline{x}))=h(\overline{x})+(1-h(\overline{x}))^2>0.
\]
Moreover, $\widehat{h}(x,z)$ has a non-zero Hessian that
$\nabla^2\widehat{h} \neq 0$. By Theorem \ref{thm:m},
\begin{equation}
(\exists\mu\in \mathbb{R})(\forall x\in \mathbb{R}^n,~\forall z\in \mathbb{R})~~
f(x) + \mu \widehat{h}(x,z)=f(x)+\mu h(x)+\mu z^2\ge0. \label{lm2}
\end{equation}
Then, $\mu\ge 0$ follows from $z\rightarrow \infty$ and setting
$z=0$ in (\ref{lm2}) yields (${\rm S_2}$). Therefore, under Slater's
condition, (${\rm S_1}$) implies (${\rm S_2}$) by Theorem
\ref{thm:m}.

Finally, when Assumption \ref{as1} holds but $B=0$, the S-lemma with
equality could possibly fail. In that case, we can formulate a
modified version of regularized S-lemma.

\begin{thm}\label{new1:3} Suppose that Assumption \ref{as1} holds but $B=0$.
The statement
\begin{equation}
h(x)=0 ~\Longrightarrow~ f (x) \ge 0,~\forall x\in\Bbb R^n \label{reg:1}
\end{equation}
is true if and only if
\begin{equation}
(\forall \epsilon>0) (\exists \lambda_{\epsilon}) (\forall x\in\Bbb R^n)
~f(x)+ \lambda_{\epsilon}(h(x))^2+ \epsilon(x^Tx + 1) \ge 0, \label{reg:2}
\end{equation}
\end{thm}
\begin{proof}
It is sufficient to prove (\ref{reg:1}) implies (\ref{reg:2}).
First, (\ref{reg:1}) implies that
\[
(h(x))^2=0 ~\Longrightarrow~ f (x) \ge 0,~\forall x\in\Bbb R^n.
\]
Since Assumption \ref{as1} cannot hold for $h^2$, we apply Theorem
\ref{new1:2} to $(f,h^2)$ and complete the proof.
\end{proof}

Theorems \ref{new1:2}, \ref{thm:m} and \ref{new1:3} can be packed
together to yield the following regularized S-lemma with equality.
\begin{cor}
The statement
\begin{equation*}
h(x)=0 ~\Longrightarrow~ f (x) \ge 0,~\forall x\in\Bbb R^n
\end{equation*}
is true if and only if
\begin{equation*}
(\forall \epsilon>0) (\exists \lambda_{\epsilon}) (\forall x\in\Bbb R^n) ~f(x)+ \lambda_{\epsilon}(h(x))^{\phi(B)}+ \epsilon(x^Tx + 1) \ge 0,
\end{equation*}
where
\[
\phi(B)=\left\{\begin{array}{ll}1,& {\textnormal ~if~B\neq0},\\
2,&{\textnormal ~if~}B=0.\end{array}\right.
\]
\end{cor}

%
%
%\begin{itemize}
%\item Suppose $B\neq 0$. According to Theorem \ref{new1:2},
%we can assume Assumption \ref{as1} holds. Then, by Theorem \ref{thm:m},
%(\ref{reg:1}) implies that there is a $\mu\in\Bbb R$ such that
%\[
%f(x)+ \mu h(x) \ge 0, \forall x\in\Bbb R^n.
%\]
%Therefore, (\ref{reg:2}) holds for any  $\epsilon>0$ and  $\lambda_{\epsilon}\equiv \mu$.
%\item Suppose $B= 0$.
%\end{itemize}

%Theorem \ref{new1:3} generalizes the regularized S-lemma \cite{J09a}. Actually, the statement (\ref{reg:in1}) is equivalent to
%\[
%h(x)+t^2=0 ~\Longrightarrow~ f (x) \ge 0,~\forall x\in\Bbb R^n.
%\]
%According to Theorem \ref{new1:3}, it is true if and only if
%\[
%(\forall \epsilon>0) (\exists \lambda_{\epsilon}) (\forall x\in\Bbb R^n,~\forall t\in\Bbb R) ~f(x)+ \lambda_{\epsilon}(h(x)+t^2)+ \epsilon(x^Tx + 1) \ge 0.
%\]
%Letting $t\rightarrow \infty$ yields $\lambda_{\epsilon}\ge0$. Then, (\ref{reg:in2}) is obtained by setting $t=0$.

\section{Application to solve (QP1EQC)}

Mor\'{e} in 1993 published an early result on the saddle point
optimality condition for (QP1EQC) under mild conditions \cite{M93}.
It states:
\begin{thm}[\cite{M93}, Thm  3.2]\label{mor}
Under Assumption \ref{as1} and  $B\neq0$, a vector $x^*$ is a global
minimizer of (QP1EQC) if and only if $h(x^*)=0$ and there is a
multiplier $\mu^*\in\Bbb R$ such that the Kuhn-Tucker condition
\begin{equation}
Ax^*+a+\mu^*(Bx^*+b)=0\label{kt}
\end{equation}
is satisfied with the second order condition $A+\mu^*B \succeq 0$.
\end{thm}
However, (\ref{kt}) may not always have a pair of solution
$(x^*,\lambda^*)$ since (QP1EQC) could be unbounded below or have an
unattainable optimal value. Even when the problem has an attainable
minimum, algorithmic approach for computing $(x^*,\lambda^*)$ from
(\ref{kt}) often required strong conditions such as the existence of
a positive definite matrix pencil $A+\mu B\succ0$ (also known as the
dual Slater condition), e.g. \cite{M93,Pong,SW,yz}. The dual Slater
condition is not practical since it is stricter than the two
matrices $A$ and $B$ being simultaneously diagonalizable via
congruence (SDC) \cite{Horn}. By the S-lemma with equality, we can
now solve (QP1EQC) directly by the standard SDP relaxation (because
it is tight) and a rank-one decomposition procedure (if necessary)
without resorting to the Kuhn-Tucker condition (\ref{kt}) and
without any assumption. We also analyze (QP1EQC) when it is
unbounded below; or is unattainable.

%Even for the
%simpler inequality version $\inf_{h(x)\le0}f(x)$, when the (SDC)
%condition is not met, there are examples that are unbounded below;
%or have an unattainable optimal value; or possess an attainable
%minimum \cite{H}.
%As an application of Theorem \ref{thm:m},
%
%we can now   We can also enhance the strong duality result (Theorem
%\ref{mor}) with a complete information to

In fact, when $B=0$, the constraint is just $2b^Tx+d=0$. When
Assumption \ref{as1} fails, there must be $B\succeq(\preceq)0$ so
that the constraint is reduced to the first order condition
$Bx+b=0$. By the null space representation of $2b^Tx+d=0$ or
$Bx+b=0$, (QP1EQC) becomes an unconstrained problem. It would then
be either unbounded below or a convex unconstrained problem with an
attainable optimal solution.

Now it remains to consider (QP1EQC) for $B\not=0$ and under
Assumption \ref{as1}. Applying the S-lemma with equality (Theorem
\ref{thm:m}), we can recast (QP1EQC) as the following semidefinite
programming problems (SDP):
\begin{eqnarray}
v({\rm QP1EQC})&=& \sup\limits_{s\in\mathbb{R}}
\left\{s:
\left\{x\in\mathbb{R}^n|f(x)-s<0,h(x)=0\right\}=\emptyset\right\}\label{SD-B}\\
&=& \sup\limits_{s\in\mathbb{R}}
\left\{s: (\exists\mu\in\mathbb{R})~
f(x)-s+\mu h(x)\ge 0,~\forall x\in\mathbb{R}^n\right\}\label{Lag}\\
&=& \sup\limits_{s,\mu\in\mathbb{R}}\left\{s:
\left[\begin{array}{cc}
A+\mu B& a+\mu b\\
a^T+\mu b^T& c+\mu d-s
\end{array}\right]
\succeq0\right\}\label{SDP1}\\
&\le & \inf\limits_{X\in S_+^n}\left\{
\left[\begin{array}{cc}
A & a \\
a^T & c \end{array}\right]\bullet X :
 \left[\begin{array}{cc}
 B&  b\\
 b^T& d
\end{array}\right]\bullet X =0, X_{n+1,n+1}=1
\right\}\label{SDP2}\\
&\le & \inf\limits_{x\in\mathbb{R}^n}\left\{
\left[\begin{array}{cc}
A & a \\
a^T & c \end{array}\right]\bullet
\left(\left[\begin{array}{cc}
x \\
 1 \end{array}\right]\left[\begin{array}{cc}
x \\
 1 \end{array}\right]^T\right):
 \left[\begin{array}{cc}
 B&  b\\
 b^T& d
\end{array}\right]\bullet \left(\left[\begin{array}{cc}
x \\
 1 \end{array}\right]\left[\begin{array}{cc}
x \\
 1 \end{array}\right]^T\right) =0
\right\} \nonumber \\
&=&v({\rm QP1EQC}).\label{SD-E}
\end{eqnarray}
Note that (\ref{Lag}) is equivalent to the Lagrangian dual of
(QP1EQC)
\begin{equation}
({\rm LD})~~~\sup_{\mu\in\mathbb{R}}\left\{ \inf\limits_{x\in\mathbb{R}^n} ~ L(x,\mu):=f(x)+ \mu h(x)
\right\}.\label{ld}
\end{equation}
The equation (\ref{SDP1}) is the SDP reformulation of (\ref{ld})
which is known as Shor relaxation scheme \cite{shor}. The inequality
(\ref{SDP2}) follows from the conic weak duality. Eventually, all
inequalities above become equalities and they prove the strong
duality (no duality gap between (QP1EQC) and its Lagrange dual), as
well as the tight SDP relaxation.

The strong duality result (\ref{SD-B})-(\ref{SD-E}) does not rely on
the existence of an optimal solution $x^*$ to (QP1EQC) and thus it
is more general than Theorem \ref{mor}. It is possible that the
strong duality holds like $v({\rm QP1EQC})=v({\rm LD})=-\infty$ for
an unbounded (QP1EQC). When $v({\rm QP1EQC})=-\infty$, the SDP
reformulation (\ref{SDP1}) of the Lagrange dual (LD) is surely
infeasible. The converse is also true. When $v({\rm
QP1EQC})>-\infty$, by the S-lemma with equality, there is some
$\mu\in\mathbb{R}$ such that $f(x)-v({\rm QP1EQC})+\mu
h(x)\ge0,~\forall x\in \mathbb{R}^n$. The dual must be feasible.

%cannot be infeasible. Otherwise, if the dual feasible set of
%(\ref{SDP1}) is empty, there would be no $\mu$ satisfying
%\begin{equation}\label{sdp:4}
%\left\{
%\begin{array}{l}
%A+\mu B \succeq 0,\\
%a+\mu b\in \mathcal{R}(A+\mu B).
%\end{array}\right.
%\end{equation}
%It follows that $\inf_{x\in\mathbb{R}^n} ~ L(x,\mu)=-\infty$ for any
%$\mu$. By the strong duality, we immediately have the following
%contradiction:
%\[
%v({\rm QP1EQC})=\sup_{\mu}\left\{ \inf_{x\in\mathbb{R}^n} ~ L(x,\mu)\right\}=-\infty.
%\]
%which evidently leads to a contradiction against the strong duality.

Moreover, when $v({\rm QP1EQC})>-\infty$, due to the tight SDP
relaxation (\ref{SDP2}), an optimal solution $x^*$ of (QP1EQC) can
be found if, and only if the primal SDP relaxation (\ref{SDP2})
attains the optimal solution at some $X^*\in S_+^n$. Then we
can employ the standard rank-one decomposition procedure \cite{sz}
to generate a rank-one solution out of $X^*$ for (QP1EQC). Notice
that the strong duality (\ref{SD-B})-(\ref{SD-E}) does not warrant
(QP1EQC) and its primal SDP relaxation (\ref{SDP2}) an attainable
optimal solution though.

We will show in Theorem \ref{thm:qp} that, when $v({\rm
QP1EQC})=v({\rm LD})>-\infty$, the dual SDP (\ref{SDP1}) is not only
feasible but the value $v({\rm LD})$ is always attainable. Moreover,
the primal problem (QP1EQC) is unattainable if and only if its dual
feasible set is a single point set $\{(v({\rm QP1EQC}),\mu^*)\}$ at
which either (\ref{soluset1}) or (\ref{soluset2}) happens. We first
prove a lemma for the dual attainment property.

\begin{lem}\label{pr:2}
Under Assumption \ref{as1} and $B\neq 0$, if (QP1EQC) has an optimal
solution $x^*$, then the dual SDP (\ref{SDP1}) also has an optimal
solution $(s^*,\mu^*)$ such that the primal-dual pair $(x^*,\mu^*)$
satisfies the Kuhn-Tucker condition (\ref{kt}).
\end{lem}
 \begin{proof}
Let $x^*$ be an optimal solution of (QP1EQC). According to Theorem
\ref{mor}, there is a $\mu^*$ such that $A+\mu^* B\succeq0$,
$a+\mu^* b\in \mathcal{R}(A+\mu^* B)$ and thus
\[
x^*={\rm arg}\min f(x)+\mu^*h(x).
\]
Let $s^*=f(x^*)+\mu^*h(x^*)=f(x^*)$. Since $s^*\le
f(x)+\mu^*h(x),~\forall x\in\mathbb{R}^n$, the pair $(s^*,\mu^*)$ is
dual feasible to (\ref{SDP1}) (or to (\ref{Lag})). Suppose $(s,\mu)$
is any dual feasible pair such that $s\le f(x)+\mu h(x),~\forall
x\in\mathbb{R}^n$. There must also be $A+\mu B\succeq0,~a+\mu b\in
\mathcal{R}(A+\mu B)$ and $x(\mu)=-(A+\mu B)^+(a+\mu b)$ such that
\begin{eqnarray*}
s&&\le \inf\{f(x)+\mu h(x): x\in\mathbb{R}^n\}\nonumber\\
&&=f(x(\mu))+\mu h(x(\mu))\nonumber\\
&&\le f(x^*)+\mu h(x^*)\nonumber\\
&&= f(x^*)\nonumber\\
&&=s^*.
\end{eqnarray*}
Therefore, the dual SDP (\ref{SDP1}) has an optimal solution
$(s^*,\mu^*)$ and the primal-dual pair $(x^*,\mu^*)$ satisfies the
Kuhn-Tucker condition (\ref{kt}).
\end{proof}

\begin{thm} \label{thm:qp}
Under Assumption \ref{as1}, $B\neq 0$ and $v({\rm QP1EQC})>-\infty$,
the dual SDP (\ref{SDP1}) always has an optimal solution
$(s^*,\mu^*)$. Moreover, the infimum of (QP1EQC) is unattainable
when, and only when the dual SDP (\ref{SDP1}) possesses a unique
feasible $\mu^*$; and at $\mu^*$ the two functions $f(x),~h(x)$
together satisfy
\begin{eqnarray}
&{\rm either}&V^TBV\succeq0,~h(y_0)-(By_0+b)^TV(V^TBV)^+V^T(By_0+b)>0,\label{soluset1} \\
&{\rm or}&V^TBV\preceq0,~h(y_0)-(By_0+b)^TV(V^TBV)^+V^T(By_0+b)<0,\label{soluset2}
\end{eqnarray}
where $y_0=-(A+\mu^* B)^+(a+\mu^* b)$, $V$ is the matrix basis  of
$\mathcal{N}(A+\mu^* B)$.
\end{thm}
\begin{proof}
Since $v({\rm QP1EQC})>-\infty$, there is some $\mu\in\mathbb{R}$
such that $f(x)-v({\rm QP1EQC})+\mu h(x)\ge0,~\forall x\in
\mathbb{R}^n$. It is clear that such an $\mu$ satisfies
\begin{equation*}%\label{sdp:4}
\left\{
\begin{array}{l}
A+\mu B \succeq 0,\\
a+\mu b\in \mathcal{R}(A+\mu B).
\end{array}\right.
\end{equation*}

We first consider that the matrix pencil $I_{\succeq}(A,B)=\{\mu: A
+\mu B\succeq0\}$ is a single-point set $\{\mu^*\}$. Then, any dual
feasible $(s,\mu^*)$ must satisfy
$$s\le\inf\{f(x)+\mu^* h(x): x\in\mathbb{R}^n\}=f(x(\mu^*))+\mu^* h(x(\mu^*))$$
where $x(\mu^*)=-(A+\mu^* B)^+(a+\mu^* b)$. Then, $(s^*,\mu^*)$ with
$s^*=f(x(\mu^*))+\mu^* h(x(\mu^*))$ is the dual optimal solution.
Obviously, $s^*=v({\rm QP1EQC})$. Moreover, when
$I_{\succeq}(A,B)=\{\mu^*\}$, all the Kuhn-Tucker points of
(\ref{kt}) can be completely specified by
\[x^*(y)=-(A+\mu^* B)^+(a+\mu^* b)+Vy=y_0+Vy,~\forall y.
 \]
Observe that
\[h(x^*(y))=h(y_0+Vy)=y^T(V^TBV)y+2(y_0^TBV+b^TV)y+h(y_0).
 \]
In case of (\ref{soluset1}), $h(x)$ restricted on the set of
Kuhn-Tucker points is convex and
 \[
\min_{y} h(x^*(y))=h(y_0)-(By_0+b)^TV(V^TBV)^+V^T(By_0+b)>0.
 \]
It indicates that the quadratic equation $h\left(x^*(y)\right)=0$
has no solution. By Theorem \ref{mor}, (QP1EQC) cannot have an
optimal solution. Since $v({\rm QP1EQC})>-\infty$, it is
unattainable. The other case (\ref{soluset2}) can be analogously
argued.

Next, we show that, if $v({\rm QP1EQC})>-\infty$ and
$I_{\succeq}(A,B)$ is not a single-point set, an optimal solution to
(QP1EQC) can be constructed.  Then, by Lemma \ref{pr:2}, we can
conclude that the dual SDP (\ref{SDP1}) is always attained when
$v({\rm QP1EQC})>-\infty$. First, by an argument in the proof of
Theorem 5.1 in \cite{M93}, $I_{\succeq}(A,B)$ is an interval with an
interior point. Denote by
\[
I_{\succeq}(A,B)=[\mu_{\min}, \mu_{\max}],~~\mu_{\min}< \mu_{\max}
 \]
whereas it is possible that $\mu_{\min}=-\infty$ and
$\mu_{\max}=+\infty$. Since $I_{\succeq}(A,B)$ is an interval with a
non-empty interior, by Theorem 3 (b) in \cite{H}, we have
 \[
 \mathcal{N}(A+\mu B)=\mathcal{N}(A)\cap \mathcal{N}(B),
 ~\forall \mu\in (\mu_{\min}, \mu_{\max}).
 \]
Let $V\in \Bbb R^{n\times r}$ be the basis matrix of
$\mathcal{N}(A)\cap \mathcal{N}(B)$ and $U\in\Bbb R^{n\times (n-r)}$
be the orthonormal complementary subspace of $V$. Then we have
\[
\left[\begin{array}{c}U^T\\V^T\end{array}\right](A+\mu B)\left[U~V\right]
=\left[\begin{array}{cc}U^TAU+\mu U^T BU& 0\\0&0\end{array}\right],
~\forall \mu\in (\mu_{\min}, \mu_{\max}).
\]
Let $u\in \Bbb R^{n-r}$. By $A+\mu B\succeq0$, $(u^TU^T)(A+\mu B)(Uu)=0$
if and only if $(A+\mu B)(Uu)=0$. Since $U$ is the orthogonal
complement of $V$, it must be $u=0$. In other words.
\begin{equation}
U^TAU+\mu U^TBU \succ 0, ~\forall \mu\in (\mu_{\min}, \mu_{\max}).
\label{UV}
 \end{equation}
With the $[U~ V]_{n\times n}$ coordinate change and the notation
$0_{m\times r}$ for the $m\times r$ zero matrix; $0_{n}$ for the
$n-$dimensional zero vector, we can recast the dual SDP (\ref{SDP1})
as
\begin{eqnarray}
&&\sup\left\{s\in\mathbb{R}:
\left[\begin{array}{cc}
U^T& 0_{n-r}\\
V^T & 0_{ r}\\
0^T_{n}& 1
\end{array}\right]\left[\begin{array}{cc}
A+\mu B& a+\mu b\\
a^T+\mu b^T& c+\mu d-s
\end{array}\right]\left[\begin{array}{ccc}
U & V& 0_{n}\\
0^T_{n-r} & 0^T_{ r}& 1
\end{array}\right]
\succeq0\right\} \nonumber\\
&=&
\sup\left\{s\in\mathbb{R}:
\left[\begin{array}{ccc}
U^TAU+\mu U^TBU& 0_{(n-r)\times r} &U^T(a+\mu b)\\
0_{r\times (n-r)}&0_{r\times r} & V^T(a+\mu b)\\
(a+\mu b)^TU&(a+\mu b)^TV&  c+\mu d-s
\end{array}\right]
\succeq0\right\} \nonumber\\
&=&
\sup\left\{s\in\mathbb{R}:
V^T(a+\mu b)=0,
\left[\begin{array}{cc}
U^TAU+\mu U^TBU&U^T(a+\mu b)\\
(a+\mu b)^TU&   c+\mu d-s
\end{array}\right]
\succeq0\right\}. \label{SDP1:eq}
\end{eqnarray}

Since $v({\rm QP1EQC})>-\infty$, the dual SDP (\ref{SDP1}) is
feasible. By (\ref{UV}), it implies that (\ref{SDP1:eq}) admits a
strict feasible point $(\mu,s)$ that satisfies the positive
semi-definite constraint. By writing down the conic dual of
(\ref{SDP1:eq}):
\begin{eqnarray}
({\rm UD})~~&\inf\limits_{Y,z}&
\left[\begin{array}{cc}
U^TAU & U^Ta \\
a^TU & c \end{array}\right]\bullet Y+a^TVz \nonumber\\
&{\rm s.t.}& \left[\begin{array}{cc}
 U^TBU&  U^Tb\\
 b^TU& d
\end{array}\right]\bullet Y+b^TVz =0,\nonumber\\
&& Y_{n-r+1,n-r+1}=1,~Y\in S_+^{n-r+1},\nonumber
\end{eqnarray}
and by the strong duality theorem, there must be
$$v({\rm QP1EQC})=v({\rm LD})=v({\rm Prob}(\ref{SDP1:eq}))=v({\rm UD})>-\infty.$$
In particular, (UD) can be attained at an optimal
solution, say $(Y^*,z^*)$. Let
$$Y^*=\left[\begin{array}{cc}
Y^*_{(n-r)\times (n-r)}& Y^*_{\{1:n-r\},n-r+1}\\
Y^{*T}_{\{1:n-r\},n-r+1} & Y^*_{n-r+1,n-r+1}
\end{array}\right]$$
where
$Y^*_{\{1:n-r\},n-r+1}=(Y^*_{1,n-r+1},Y^*_{2,n-r+1},\cdots,Y^*_{n-r,n-r+1})^T$.
Then, define
\[
X^*=\left[\begin{array}{ccc}
U & V& 0_{n}\\
0^T_{n-r} & 0^T_{ r}& 1
\end{array}\right]\left[\begin{array}{ccc}
Y^*_{(n-r)\times (n-r)} & 0_{(n-r)\times r}  & Y^*_{\{1:n-r\},n-r+1}  \\
0_{r\times (n-r)}&\frac{1}{4}z^*z^{*T} &\frac{1}{2}z^*\\
Y^{*T}_{\{1:n-r\},n-r+1} & \frac{1}{2}z^{*T} &1 \end{array}\right]\left[\begin{array}{cc}
U^T& 0_{n-r}\\
V^T & 0_{ r}\\
0^T_{n}& 1
\end{array}\right]\succeq0.
\]
We can verify that $$\left[\begin{array}{cc}
 B&  b\\
 b^T& d
\end{array}\right]\bullet X^* =0,~~X^*_{n+1,n+1}=1,$$
and
$$\left[\begin{array}{cc}
A & a \\
a^T & c \end{array}\right]\bullet X^*=\left[\begin{array}{cc}
U^TAU & U^Ta \\
a^TU & c \end{array}\right]\bullet Y^*+a^TVz^*=v({\rm UD})=v({\rm
QP1EQC}).$$ In other words, $X^*$ is an optimal solution of the
primal SDP (\ref{SDP2}). Employing the standard rank-one
decomposition procedure \cite{sz}, we have shown that (QP1EQC) is
attained. The proof is complete.
\end{proof}

The following example is provided to illustrate the idea of Theorem
\ref{thm:qp}.
\begin{exam}
Consider
\begin{eqnarray*}
 &\inf& x_1^2\\
&{\rm s.t.} & x_1x_2-1=0,
\end{eqnarray*}
where $a=b=(0,0)^T$, $c=0$, $d=-1$ and
\begin{equation*}
A=\left[\begin{array}{cc}
1 & 0 \\
0  & 0\end{array}\right],~B=\left[\begin{array}{cc}
0 & 0.5 \\
0.5  & 0\end{array}\right].
\end{equation*}
The graph of the objective function is a paraboloid by extending the
parabola $z=x_1^2$ parallel along the $x_2$ direction. The
constraint is a hyperbola upon which there are a pair of traces on
$z=x_1^2$, both asymptotically approaching 0 from the positive
territory. Then, $v({\rm QP1EQC})=0$ but is unattainable.

The primal SDP (\ref{SDP2}) is
\begin{eqnarray*}
 &\inf& \left[\begin{array}{ccc}
1 & 0 & 0 \\
0&0 & 0\\
0&0&0\end{array}\right]\bullet X \\
&{\rm s.t.} &
\left[\begin{array}{ccc}
0 & 0.5 & 0 \\
0.5&0 & 0\\
0&0&-1\end{array}\right]\bullet X =0,\\
&& X_{3,3}=1, ~X\succeq0.
\end{eqnarray*}
Define
\[
X(\epsilon)=\left[\begin{array}{ccc}
\epsilon& 1 & 0 \\
1&\frac{1}{\epsilon} & 0\\
0&0&1\end{array}\right].
\]
Then $X(\epsilon)$ is a feasible solution of the primal SDP
(\ref{SDP2}) for any $\epsilon>0$. The corresponding objective value
is $\epsilon$. Since $\lim_{\epsilon\rightarrow 0}X(\epsilon)$ does
not have a limit, the optimal value of the primal SDP is $0$, but
unattainable.

The  dual SDP (\ref{SDP1}) is
\begin{eqnarray*}
 &\sup&s\\
&{\rm s.t.} &  \left[\begin{array}{ccc}
1 & 0.5\mu & 0 \\
0.5\mu&0 & 0\\
0&0&-\mu-s \end{array}\right]\succeq0.
\end{eqnarray*}
Obviously, $\mu^*=0$ is the only feasible dual solution, which
forces $s\le0$. Then, the dual optimal value is $0$, confirming the
strong duality; and the dual optimal solution is
$(s^*,\mu^*)=(0,0)$, confirming that the dual must be attainable if
the problem is bounded from below.

Moreover,  we can verify that
\[
\mathcal{N}(A+\mu^* B)=\mathcal{N}(A)={\rm span}\{(0,1)^T\}
 \]
such that $V=(0,1)^T$. A direct computation shows
\[
 y_0=-(A+\mu^* B)^+(a+\mu^* b)=(0,0)^T;~V^TBV=0,
 \]
and hence
 \[
h(y_0+Vy)=h(y_0)-(By_0+b)^TV(V^TBV)^+V^T(By_0+b)=-1,
 \]
which confirms (\ref{soluset2}) that none of the Kuhn-Tucker points
are feasible.
\end{exam}

%\noindent{\bf Interval bounded generalized trust region subproblem (GTRS).}
Finally, we consider the interval bounded generalized trust region subproblem (GTRS),
%The problem (GTRS)
which is to deal with $\inf\{f(x):  l\le
h(x)\le u,~x\in\mathbb{R}^n\}.$ Jeyakumar et al. \cite{J09b}
studied (GTRS) using Polyak's result \cite{Poly} (see also Theorem
\ref{thm:poly} in Section 5). They assumed that $h$ is homogeneous
and strictly convex. Pong and Wolkowicz \cite{Pong} proved the
strong duality based on the the following assumptions:
\begin{itemize}
\item[1.] $B \neq 0$.
\item[2.] (GTRS)  is feasible.
\item[3.] The following relative interior constraint qualification holds
\[
{\rm(RICQ)}~  l< B\bullet\widehat{X}+2b^T\widehat{x}+d<u,~ {\rm for~ some}~ \widehat{X} \succ \widehat{x}\widehat{x}^T.
\]
\item[4.] (GTRS) is bounded below.
\item[5.] The dual problem of (GTRS) is feasible.
\end{itemize}
Very recently, Wang and Xia \cite{Shu} showed that Assumption 3
above in \cite{Pong} can be equivalently rephrased as a simple
``strict feasibility'' assumption that there exists an $\overline{x}\in
\mathbb{R}^n$ such that $l< h(\overline{x})< u.$ Hence Assumption 2
in \cite{Pong} is redundant. In addition, Assumption 4 naturally
implies Assumption 5 by a similar result from Theorem \ref{thm:qp}
in this section. This has also been done in \cite{Shu} by Wang and
Xia. Indeed, based on the S-lemma with equality from an earlier
arXiv version of this paper, they established the S-lemma with
interval bounds. Under the strict feasibility assumption, the
following two statements are equivalent ((${\rm I_1}$)$\sim$(${\rm
I_2}$)):
\begin{itemize}
\item[(${\rm I_1}$)] $l \leq h(x)\leq u~\Longrightarrow~ f(x)\ge 0,
~\forall x\in\Bbb R^n.$ %\label{slm}
\item[(${\rm I_2}$)] ($\exists~\mu\in \mathbb{R}$)~
$f(x) + \mu_{-} (h(x)-u)+\mu_{+}(l-h(x))\ge0, ~\forall x\in \mathbb{R}^n$ where $\mu_{+}=\max\{\mu,0\},\mu_{-}=-\min\{\mu,0\}.$ %\label{lm}
\end{itemize}
except for the special case where $A$ has exactly one negative
eigenvalue, $B=0$, $b\neq 0$ and there exists a $\nu\geq0$ such that
\begin{equation}
\left[\begin{array}{ccc}
V^TAV                  & \frac{1}{2b^Tb}V^TAb          & V^Ta\\
\frac{1}{2b^Tb}b^TAV   & \frac{b^{T}Ab}{(2b^Tb)^2}+\nu &
\frac{a^{T}b}{2b^Tb}-\frac{\nu}{2}(l+u-2d)\\
a^TV & \frac{a^{T}b}{2b^Tb}-\frac{\nu}{2}(l+u-2d) &
c+\nu(l-d)(u-d)
\end{array}\right]\succeq 0, \label{pd}
\end{equation}
with $V\in \Bbb R^{n\times (n-1)}$ being the matrix basis of
$\mathcal{N}(b)$.

In summary, (GTRS) is now completely answered by the S-lemma with
equality. When the strict feasibility holds and $B\not=0$, (GTRS)
has a tight SDP relaxation (the SDP formulation can be read in
\cite{Pong,Shu}). In that case, if $v({\rm GTRS})>-\infty$ and
attainable, an optimal solution can be found by solving the SDP
relaxation followed by a typical rank-one decomposition procedure.
Otherwise, when (GTRS) fails to satisfy the strict feasibility or
has $B=0$, $h(x)$ reduces to be linear. Then, (GTRS) would be either
a convex unconstrained optimization problem with its optimal
solution residing in the interval $l\le h(x)\le u$; or can be
determined by $v({\rm GTRS})=\min\left\{\inf_{h(x)=l}
f(x),~\inf_{h(x)=u} f(x) \right\}.$ Since $h(x)=l$ (or $h(x)=u$) can
be replaced by a system of linear equations, each of both becomes an
unconstrained problem and could either be unbounded below or has an
attainable optimal solution.

 %The S-lemma
%with interval bounds guarantees  the strong duality for  (GTRS)
%under the assumption that there exists an $\overline{x}\in
%\mathbb{R}^n$ such that $l< h(\overline{x})< u$ and $B\neq 0$.
%
%\textcolor{red}{ It was shown that if one of the Items 1, 2, 3 in
%the above Assumption fails, then an explicit solution of (GTRS) can
%easily be obtained. Moreover, if Items 1, 2, 3 in the above
%Assumption hold and  $ b = 0$, then Item 4 implies Item 5. However,
%it was unknown whether Item 4 implies Item 5 when $b\neq 0$, see
%Remark 2.2 \cite{Pong}. In our very recent technical note
%\cite{Shu}, the assumption there exists an $\overline{x}\in
%\mathbb{R}^n$ such that $l< h(\overline{x})< u$ is shown to be
%equivalent to the above Item 3. It means that, Item 2 is no longer
%needed and only Items 1 and 3 are necessary to guarantee the strong
%duality. Consequently, Pong and Wolkowicz's above open question is
%positively answered. }

\section{Application to Convexity of Joint Numerical Range}

%Yakubovich's proof of the S-lemma with inequality relied on a result
%of

Dines in 1941 proved a fundamental but somehow surprising result in
classical mathematical analysis that the joint numerical range
$M=\{(f(x),h(x)): x\in \Bbb R^n\}\subseteq \Bbb R^2$ is convex when
$f$ and $h$ are quadratic forms \cite{Dines}. He explained in the
same paper that the observation on the convexity of $M$ was
motivated by Finsler's theorem (S-lemma) \cite{F} in 1937 because
``{\it it asserts the existence of a supporting line of the map $M$
which has contact with $M$ only at $(0,0)$. This suggests that the
theorem is related to the theory of convex sets.}'' Dines also
described the shape of $M$ to be either the entire $x\text{-}y$ plane or an
angular sector of angle less than $\pi$, provided $f, h$ have no
common zero except $x=0$.

Since then, the progress has been slow. Extension of Dines' result
to the 2D image of nonhomogeneous functions $f$ and $h$ occurred
much later in 1998 due to Polyak \cite{Poly}. It may be stated as

\begin{thm}[\cite{Poly}]\label{thm:poly}
Let $f, h \in \Bbb R^n$ be nonhomogeneous quadratic functions.
Suppose that $n\ge 2$ and there exists $\mu\in\Bbb R^2$ such that
\begin{equation*}
\mu_1A+\mu_2B\succ 0.%\label{cov:pos}
\end{equation*}
Then the set $M= \{(f(x), h(x)): x\in \Bbb R^n\}\subseteq \Bbb R^2$
is closed and convex.
\end{thm}

A counterexample was provided in the same paper \cite{Poly} to show
that the joint numerical range $M$ is in general nonconvex for
nonhomogeneous quadratic functions. Dines' theorem even fails when
one of $f$ and $h$ is an affine function.

\begin{exam}[\cite{Poly}]\label{exam:cv}
Consider the following two quadratic functions in $\Bbb R^2$
\[
f(x)=2x_1^2-x_2^2,~h(x)=x_1+x_2.
\]
Then we can verify that
\[
M= \{(f(x), h(x)): x\in \Bbb R^n\}=\{(y_1,y_2)\in\Bbb R^2: y_1\ge -2y_2^2\}
\]
which is nonconvex.
\end{exam}

In 2007, Beck \cite{BE07} studied the convexity of the image of
mappings comprised of a strictly convex quadratic function and a set
of affine functions.
\begin{thm}[\cite{BE07}]\label{thm:beck}
Let $h_{i}(x)=2b_{i}^Tx+d_{i}$ for $i=1,\ldots,p$, where $b_{i}\in
\Bbb R^n, d_{i}\in \Bbb R$. Suppose $A\succ 0$. When $p\le n-1$, the
set
\begin{equation}
S=\{(f(x),h_{1}(x),\ldots,h_{p}(x)):x\in\Bbb R^n\}\subseteq \Bbb R^{p+1}\label{Mset}
\end{equation}
is closed and convex. Besides, when $p=n$ and $b_{1},\ldots,b_n$ are
linearly independent, $S$ is nonconvex.
\end{thm}

It happens that the newly developed S-lemma with equality can be
used to give a theorem on the convexity of jointly numerical range
$S$, sufficiently strong to cover both Example \ref{exam:cv} and
Theorem \ref{thm:beck}. Namely, we are able to say something similar
to what Dines had done in \cite{Dines}, from the S-lemma with
equality ({\it cf.} Finsler's theorem) to the convexity of $S$ ({\it
cf.} $M$).

%
% \textcolor{red}{and more
%general than Beck's results \cite{BE07}}  from the S-lemma to

\begin{thm}\label{thm:cv}
Suppose $h_{i}(x)=2b_{i}^Tx+d_{i}$ for $i=1,\ldots,p$, where
$b_{i}\in \Bbb R^n, d_{i}\in \Bbb R$. Let $P=[b_{1}, b_{2}, \ldots,
b_{p}]^T$ and $r={\rm rank}(P)$. Then, the set $S$ defined in
(\ref{Mset}) is convex if and only if neither of the following two
cases occurs:
\begin{itemize}
\item[${\rm(a)}$]$V^TAV\succeq 0$, $V^Ta\in \mathcal{R}(V^TA)$ and $W^TAW$ has a negative eigenvalue;
\item[${\rm(b)}$]$V^TAV \preceq 0$, $V^Ta\in \mathcal{R}(V^TA)$ and $W^TAW$ has a positive eigenvalue,
\end{itemize}
where $V\in \Bbb R^{n\times (n-r)}$ is the matrix basis of
$\mathcal{N}(P)$ and $W\in \Bbb R^{n\times (n-{\rm rank}(V^TA))}$ is
the  matrix basis of $\mathcal{N}(V^TA)$.
\end{thm}
\begin{proof}
Let $u=(u_f,u_{h_{1}},\ldots,u_{h_{p}})$ and
$v=(v_f,v_{h_{1}},\ldots,v_{h_{p}})$ be any two distinct points in
$S$. That is, there exists $x_u, x_v\in\Bbb R^n$ and $x_u\neq x_v$
such that
\[
u_f=f(x_u),~u_{h_{i}}=h_{i}(x_{u});~v_f=f(x_v),~v_{h_{i}}=h_{i}(x_v),~i=1,\ldots,p.
\]
Then, $S$ is nonconvex if and only if we cannot find
$x_\lambda\in\Bbb R^n$ such that
$$(f(x_\lambda),h_{1}(x_\lambda),\ldots,h_{p}(x_\lambda))=(1-\lambda)u+\lambda
v,~\lambda\in(0,1).$$ Equivalently, the following system in terms of
$(x_{\lambda}, \lambda)$ has no solution
\begin{eqnarray}
&& f(x_{\lambda}) =(1-\lambda)u_f+\lambda v_f,
\label{eq:fa}\\
&& h_{i}(x_{\lambda}) =(1-\lambda)u_{h_{i}}+\lambda v_{h_{i}},~i=1,\ldots,p,
\label{eq:ha}\\
&&\lambda\in (0, 1).\label{eq:lbd}
\end{eqnarray}
Since $h_1,\ldots,h_p$ are affine, the equations (\ref{eq:ha}) imply that% for each $i\in\{1,\ldots,p\}$, it holds that
\begin{eqnarray*}
h_{i}(x_{\lambda})&=& 2b_{i}^T x_{\lambda}+d_{i} \\
&=& (1-\lambda)(2b_{i}^Tx_u+d_{i})+\lambda(2b_{i}^Tx_v+d_{i})\\
&=&2b_{i}^T((1-\lambda)x_u+\lambda x_v)+d_{i},~i=1,2,\ldots,p
\end{eqnarray*}
and thus $x_{\lambda}$ has to lie on the affine space:
\[
x_{\lambda} = (1-\lambda)x_u+\lambda x_v +Vy,~{\rm for ~some}~y\in\Bbb R^{n-r},
\]
where $V\in \Bbb R^{n\times (n-r)}$ is the matrix basis of
$\mathcal{N}\left([b_{1},b_{2},\ldots,b_{p}]^T\right)$.
Substituting $x_{\lambda}$ into (\ref{eq:fa}), we obtain
\begin{eqnarray}
f(x_{\lambda})&=&f\left((1-\lambda)x_u+\lambda x_v +Vy\right)\nonumber\\
&=& (1-\lambda)^2x_u^TAx_u+\lambda^2x_v^TAx_v+2\lambda(1-\lambda) x_u^TAx_v+
 2a^T\left((1-\lambda)x_u+\lambda x_v \right)\nonumber\\ &~~&+y^TV^TAVy+2
 \left((1-\lambda)x_u+\lambda x_v \right)^TAVy+2a^T Vy+c.\label{f-lambda}
\end{eqnarray}
Equate (\ref{f-lambda}) with
$$(1-\lambda)f(x_u)+\lambda f(x_v)=(1-\lambda)(x_u^TAx_u)+\lambda x_v^TAx_v
+2a^T\left((1-\lambda)x_u+\lambda x_v \right)+c$$ to yield
\begin{eqnarray*}
&&y^TV^TAVy +2\left((1-\lambda)x_u+\lambda x_v \right)^TAVy+2a^TVy\\
&=&\lambda(1-\lambda)(x_u^TAx_u+ x_v^TAx_v-2x_u^TAx_v).
\end{eqnarray*}
This is a quadratic equation in variables $(\lambda,y)\in\Bbb
R^{1+n-r}.$ To simplify, let
$\delta=x_u^TAx_u+x_v^TAx_v-2x_u^TAx_v=(x_u-x_v)^TA(x_u-x_v)$ and
\begin{equation}
G=\left[\begin{array}{cc}  \delta& (x_v-x_u)^TAV \\
 V^TA(x_v-x_u)&  V^TAV
\end{array}\right],~q=\left[\begin{array}{c}  -\delta\\ 2V^T(Ax_u+a)
\end{array}\right]\label{G-p}
\end{equation}
to obtain
\begin{eqnarray*} g(\lambda,y) &=&\delta
\lambda^2-\left(2(x_u-x_v)^TAVy+\delta\right)
\lambda+y^TV^TAVy +2x_u^TAVy+2a^TVy\\
&:=&\left[\begin{array}{c}  \lambda\\ y
\end{array}\right]^TG \left[\begin{array}{c}  \lambda\\ y
\end{array}\right]+q^T\left[\begin{array}{c}  \lambda\\ y
\end{array}\right].
\end{eqnarray*}

With all the settings, the joint numerical range $S$ is not convex
if and only if the system (\ref{eq:fa})-(\ref{eq:lbd}) is
unsolvable, which is equivalent to the system
\begin{equation}
 \lambda^2-\lambda<0,~g(\lambda,y)=0  \label{eq:gcv}
\end{equation}
having no solution.

We can show that, when (\ref{eq:gcv}) is unsolvable, $g(\lambda,y)$ must take both
positive and negative values. Should such a claim fail, we would have
\begin{equation}
 G\succeq (\preceq)0,~
 q\in \mathcal{R}\left(G\right),
 q^TG^+q=0. \label{p00}
\end{equation}
Since $G\succeq (\preceq)0$ implies that $G^+\succeq (\preceq)0$,
$q^TG^+q=0$ implies that $G^+q=0$. Therefore,
\[
q\in \mathcal{N}(G^+)= \mathcal{N}(G).
\]
However, from (\ref{p00}), we also known $q\in \mathcal{R}(G)$ and
thus $q=0$. By (\ref{G-p}),
\[
\delta=0,~  V^T(Ax_u+a)=0.
\]
Since $G\succeq (\preceq)0$, all $2\times 2$ principal minors
consisting of $\delta=0$ must have non-negative determinants. Then,
the first row (column) of $G$ should be identically 0. Namely,
$V^TA(x_v-x_u)=0$. The function $g$ is thus reduced to
$g(\lambda,y)=y^TV^TAVy$ so that (\ref{eq:gcv}) would be trivially
solvable. It is a contradiction.

Since Assumption 1 holds and since the quadratic term of the
function $\varphi(\lambda,y)=\lambda^2-\lambda$ does not have
exactly one negative eigenvalue, by the S-lemma with equality, the
statement (\ref{eq:gcv}) has no solution if and only if there is a
$\mu$ such that
\[
 \lambda^2-\lambda +\mu g(\lambda,y)\ge 0,~\forall \lambda\in\Bbb R,~y\in\Bbb R^{n-r},
\]
which has the following positive semi-definite formulation:
\begin{equation}
\left[\begin{array}{ccc} 1+\mu \delta& -\mu(x_u-x_v)^TAV& -\frac{1}{2}-\frac{1}{2}\mu\delta\\
-\mu V^TA(x_u-x_v)& \mu V^TAV& \mu V^T(Ax_u+a)\\
-\frac{1}{2}-\frac{1}{2}\mu\delta & \mu (x_u^TA+a^T)V& 0
\end{array}
\right]  \succeq 0. \label{eq:cvge}
\end{equation}
By considering all $2\times 2$ principal minors including the $0$
element in (\ref{eq:cvge}), we obtain
\begin{equation*}
-\frac{1}{2}-\frac{1}{2}\mu\delta = 0,~\mu V^T(Ax_u+a)=0,~\mu V^TA(x_u-x_v)=0,~\mu  V^TAV\succeq  0.
\end{equation*}
Therefore, $\mu\neq 0,~\delta\not=0$ and
\begin{eqnarray}
&&V^Ta\in \mathcal{R}(V^TA),\label{cveq:3}\\
&& x_u-x_v \in \mathcal{N}(V^TA),\label{cveq:1}\\
&& -\frac{1}{\delta}V^TAV\succeq  0.\label{cveq:2}
\end{eqnarray}
Let $W\in \Bbb R^{n\times (n-{\rm rank}(V^TA))}$ be the  matrix
basis of $\mathcal{N}(V^TA)$ and express (\ref{cveq:1}) as
\[
x_u-x_v = Wz,~{\rm for~some}~z\in \Bbb R^{n-{\rm rank}(V^TA)}.
\]
Then, (\ref{cveq:2}) implies that
\[
\delta=(x_u-x_v)^TA(x_u-x_v)=z^TW^TAWz\left\{\begin{array}{cc}
<0,&{\rm if}~V^TAV\succeq  0, \\
>0,&{\rm if}~V^TAV\preceq  0, \\
\end{array}\right.
\]
which, together with (\ref{cveq:3}), completes the proof.
\end{proof}

By Theorem \ref{thm:cv}, we can now confirm that the joint numerical
range in Example \ref{exam:cv} is non-convex.
%
%
%  by the following example
%\textcolor{red}{where $p=1$}:
\begin{exam}
Consider the nonconvex Example \ref{exam:cv}  where  $n=2,~p=1$:
\[
A=\left[\begin{array}{cc}2&0\\0&-1\end{array}\right],~
a=\left[\begin{array}{c}0\\0\end{array}\right],~
b=\left[\begin{array}{c}\frac{1}{2}\\\frac{1}{2}\end{array}\right].
\]
Then we can compute
\[
V=\left[\begin{array}{c}\frac{\sqrt{2}}{2}\\-\frac{\sqrt{2}}{2}\end{array}\right],~
V^TA=\left[\begin{array}{c}\sqrt{2}~\frac{\sqrt{2}}{2}\end{array}\right],~
W=\left[\begin{array}{c}\frac{\sqrt{5}}{5}\\-\frac{2\sqrt{5}}{5}\end{array}\right]
\]
and, by Case (a) in Theorem \ref{thm:cv},
$M=\{(2x_1^2-x_2^2,x_1+x_2): (x_1,x_2)\in\Bbb R^2\}$ is
nonconvex since
\[
V^TAV=\frac{1}{2}\succeq 0,~ V^Ta=0\in \mathcal{R}(V^TA),~
W^TAW= -\frac{2}{5}.
\]
\end{exam}

Next, we show that Theorem \ref{thm:cv} implies
Theorem \ref{thm:beck} under $A\succ 0$.
\begin{itemize}
\item Suppose $p\le n-1$. We have
\begin{equation}
r={\rm rank}\left([b_{1}, b_{2}, \ldots, b_{p}]^T\right)\le n-1. \label{ibeck}
\end{equation}
Then, ${\rm rank}(V)=n-r\ge 1.$ Since $V$ is a basis matrix
consisting of at least one column and $A\succ 0$, we have
$V^TAV\succ 0$ and $W^TAW\succeq 0$. It follows that neither
Case (a) nor Case (b) in Theorem \ref{thm:cv} holds. Therefore,
the set $S$ defined in (\ref{Mset}) is convex.  We remark that
the condition $p\le n-1$ in Theorem \ref{thm:beck} can be
improved to (\ref{ibeck}).
%
%\[
%V^TAV\succ 0, ~V^Ta\in \mathcal{R}(V^TA),~W^TAW\succeq 0,
%\]

\item Suppose $p=n$ and $b_{1},\ldots, b_n$ are
linearly independent. We have $r={\rm rank}\left([b_{1}, b_{2},
\ldots, b_{p}]^T\right)=n$, and ${\rm rank}(V)=0,~{\rm
rank}(W)=n.$ Then,
\[
V^TAV=0, ~V^Ta\in \mathcal{R}(V^TA),~W^TAW\succ 0.
\]
By Case (b) in Theorem \ref{thm:cv}, the set $S$ defined in
(\ref{Mset}) is nonconvex.
\end{itemize}

Finally in this section, we mention a weaker result than the
convexity of $\{(f(x),h_{1}(x),\ldots,h_{p}(x)):x\in\Bbb R^n\}$,
which asks for the convexity of
\begin{equation}
\{(f(x),h_{1}(x),\ldots,h_{p}(x)):x\in\Bbb R^n\} + \Bbb R_+^{p+1},\label{Sp}
\end{equation}
where $f,h_1,\ldots,h_p$ are nonhomogeneous quadratic functions and
$\Bbb R_+^{p+1}$ is the nonnegative orthant of $\Bbb R^{p+1}$. For
the special case $p=1$, Tuy and  Tuan showed that
$\{(f(x),h_{1}(x)):x\in\Bbb R^n\}+ \Bbb R_+^{2}$ is always convex,
see Corollary 10 of \cite{TT}. For a general $p$,  suppose
$f(x),h_1(x),\ldots,h_p(x)$ are homogeneous quadratic functions and
$\nabla^2f(x),\nabla^2h_1(x),\ldots,\nabla^2h_p(x)$ are all
Z-matrices (A real symmetric matrix $A$ is called a Z-matrix if
$A_{ij}\le 0$ for all $i\neq j$), Jeyakumar et al. \cite{J09b}
showed that
\[
\{(f(x),h_{1}(x),\ldots,h_{p}(x)):x\in\Bbb R^n\} + {\rm int}\Bbb R_+^{p+1}
\]
is convex, where ${\rm int}\Bbb R_+^{p+1}$ is the interior of $\Bbb R_+^{p+1}$. More recently, suppose $h_1(x)$ is a strictly convex quadratic function, $h_2(x),\ldots,h_p(x)$ are affine linear functions, for any nonhomogeneous quadratic function $f(x)$, Jeyakumar and Li \cite{J13}
showed that the set $(\ref{Sp})$ is convex under the dimension assumption:
\[
{\rm dim} \mathcal{N}\left(\nabla^2f(x)-\lambda_{\min}(\nabla^2f(x))I_n\right)
\ge {\rm dim}  {\rm~span}\{\nabla h_2(x),\ldots,\nabla h_p(x)\}+1,
\]
where ${\rm dim} L$ denotes the dimension of the subspace $L$ and $\lambda_{\min}(A)$ is the minimal eigenvalue of the matrix $A$.

Since the convexity of $S$ guarantees the convexity of $S+\Bbb
R_+^{p+1}$, Theorem \ref{thm:cv} trivially implies a new sufficient
condition for the set $(\ref{Sp})$.
\begin{cor}
Let $f(x)=x^TAx+2a^Tx+c$ and $h_{i}(x)=2b_{i}^Tx+d_{i}$ for
$i=1,\ldots,p$. Suppose neither Case (a) nor Case (b) in Theorem
\ref{thm:cv} occurs, the set
$\{(f(x),h_{1}(x),\ldots,h_{p}(x)):x\in\Bbb R^n\} + \Bbb R_+^{p+1}$
is convex.
\end{cor}
Moreover, when $p=1$ and $h_1(x)$ is an affine function, with the
help of Theorem \ref{thm:cv}, we can see  how
$\{(f(x),h_{1}(x)):x\in\Bbb R^n\} + \Bbb R_+^{2}$ becomes convex in
the case that $\{(f(x),h_{1}(x)):x\in\Bbb R^n\}$ is nonconvex.
\begin{cor}
Let $f(x)=x^TAx+2a^Tx+c$ and $h_{1}(x)=2b_{1}^Tx+d_{1}$. Suppose
$S:=\{(f(x),h_{1}(x)):x\in\Bbb R^n\}$ is nonconvex. Then, exactly
one of the following cases happens:
\begin{itemize}
\item[${\rm(i)}$] The set $S$ contains a parametric curve
$C(t)=\{(u(t),v(t)):t\in\Bbb R\}$ satisfying
$\lim\limits_{t\rightarrow \infty}u(t)=\lim\limits_{t\rightarrow
\infty}v(t)=-\infty$. It follows immediately that $S + \Bbb
R_+^{2}= \Bbb R^2.$
\item[${\rm(ii)}$] $A=\alpha b_{1}b_{1}^T$ for some $\alpha>0$. Consequently,
$S+ \Bbb R_+^{2}$ is convex.
\end{itemize}
\end{cor}
\begin{proof}
Since $S$ is nonconvex, by Theorem \ref{thm:cv}, either Case (a) or
Case (b) occurs.

Suppose Case (a) in Theorem \ref{thm:cv} happens. Then $A$ has a
negative eigenvalue. Let $z\neq 0$ be the corresponding eigenvector,
i.e., $z^TAz<0$. Then, $b_1^Tz\neq 0$. Otherwise, $b_1^Tz=0$ would
imply that $z\in \mathcal{R}(V)$. It would follow from $V^TAV\succeq
0$ that $z^TAz\ge 0$, which is a contradiction. Without loss of
generality, we assume that $b_1^Tz<0$ and define the parametric
curve
$$C(t)=\{(u(t),v(t)):~u(t)=f(tz),~v(t)=h_1(tz),~t\in\Bbb R\}.$$
Then, $C(t)\subset S$ and
\begin{eqnarray}
&&\lim_{t\rightarrow\infty} u(t)=\lim_{t\rightarrow\infty} f(tz)=\lim_{t\rightarrow\infty}
t^2z^TAz+2ta^Tz+c = -\infty; \label{wk1}\\
&&\lim_{t\rightarrow\infty} v(t)=\lim_{t\rightarrow \infty} h_1(tz)=
\lim_{t\rightarrow\infty} 2t(b_1^Tz)+d_{1}= -\infty\label{wk2},
\end{eqnarray}
which shows that (i) holds.

Suppose Case (b) in Theorem \ref{thm:cv} happens. We know $A$ must
have a positive eigenvalue with $z\neq 0$ being the corresponding
eigenvector. Since $V^TAV\preceq0$ and $z^TAz>0$, there must be
$b_1^Tz\neq 0$ and we again assume that $b_1^Tz<0$. In addition, if
$A$ also has a negative eigenvalue except for the positive one(s),
by the same argument as in (\ref{wk1})-(\ref{wk2}), we conclude
immediately that (i) holds.

Otherwise, when $A\succeq0$ but $V^TAV\precneqq 0$, there exists a
vector $u$ such that $u^TV^TAVu<0$. By continuity, we can find a
sufficiently small $\beta>0$ such that $(Vu+\beta z)^TA(Vu+\beta
z)<0$. By defining the curve
$$C(t)=\{(u(t),v(t)):~u(t)=f(t(Vu+\beta z)),~v(t)=h_1(t(Vu+\beta z)),~t\in\Bbb R\},$$
we also see that (i) holds since
\begin{eqnarray*}
&&\lim_{t\rightarrow\infty} u(t) \\
=&&\lim_{t\rightarrow\infty}
t^2(Vu+\beta z)^TA(Vu+\beta z)+2t a^T(Vu+\beta z)+c \\
=&& -\infty,
\end{eqnarray*}
and
\begin{eqnarray*}
\lim_{t\rightarrow\infty} v(t)=\lim_{t\rightarrow\infty} h_1(t(Vu+\beta z))=
\lim_{t\rightarrow\infty} 2t \beta(b_1^Tz)+d_{1}= -\infty.
\end{eqnarray*}

Finally, it remains to show that $A\succeq 0$ and $V^TAV=0$ lead to
(ii). When this happens, we have $AV=0$. Since $V$ is the matrix
basis for the null space of the $1\times n$ matrix $[b_1^T]$, $A$
must be of rank one with the form $A=\alpha b_1b_1^T$ for some
$\alpha>0$. That is, (ii) holds. It follows from the convexity of
$f(x)$ and $h_1(x)$ that $S+\Bbb R_+^{2}$ is convex.
\end{proof}

\section{Concluding remarks}
This paper is devoted to a completely new understanding toward the
S-lemma with equality. While the inequality version (${\rm
S_1}$)$\sim$(${\rm S_2}$) has been established under Slater's
condition, it was not immediately clear whether the equality version
(${\rm E_1}$)$\sim$(${\rm E_2}$) could be a real obstacle. While the
inequality version (${\rm S_1}$)$\sim$(${\rm S_2}$) has been widely
used in many applications, the important consequence of (${\rm
E_1}$)$\sim$(${\rm E_2}$), perhaps due to lack of an affirmative
result, was not yet visualized before. As (${\rm E_1}$)$\sim$(${\rm
E_2}$) is not true in general, we really need to get down to the
most subtle detail looking for a successful argument, many from
geometrical observations on quadratic manifolds. This complicated
essence as well as the poor accessibility to intuition make it hard
to come out with an easy-to-grasp intuitive proof. The long-standing
interval bounded generalized trust region subproblem (GTRS) has now
been resolved thoroughly by the full characterization of the
quadratic programming with a single quadratic equality constraint
(QP1EQC).
%There will be no easy cases or hard cases anymore.
The relation between the S-lemma and the convexity of joint
numerical ranges is now further strengthened, indicating a step
forward to the duality theory for nonconvex optimization. We wish
that the study can sparkle new idea for
solving %other hard problems in nonconvex quadratic optimization programming.
the
nonconvex quadratic optimization problem with multiple constraints and polynomial optimization problems.

\section*{Appendix}

We discuss the relations among S-Conditions 1,2,3 and 4. First, it
is easy to see that the definiteness of $B$ implies $A \succeq\eta
B$ for some $\eta.$ Therefore,
\[
{\rm S\text{-}Condition\ 1} ~\Longrightarrow~{\rm S\text{-}Condition\ 2}.
\]
Moreover, when $B$ is definite, $x^TBx=0$ if and only if $x=0$ and
thus
\[
{\rm S\text{-}Condition\ 1} ~\Longrightarrow~{\rm S\text{-}Condition\ 4}.
\]
When $h(x)$ is homogeneous, there is $b=d=0$ so that $h(0)=0.$ By
choosing $\zeta=0$,
\[
{\rm S\text{-}Condition\ 3} ~\Longrightarrow~{\rm S\text{-}Condition\ 4}
\]

It is not difficult to verify that neither S-Condition 2 nor
S-Condition 4 can imply each other \cite{Bong}. Consequently,
neither S-Condition 2 nor S-Condition 4 is necessary for the S-lemma
with equality.

%Notice that, when both $f$ and $h$ are quadratic forms (i.e.,
%$a=b=0,~c=d=0$), Assumption \ref{as2} can be removed but Assumption
%\ref{as1} is still needed. For this special case, a short proof
%based on Finsler's Theorem can be found in \cite{M93}.
%Interestingly, the same result in Hilbert spaces over the complex
%field was established much earlier in 1969 by Krein and Smuljan
%\cite{K69}.

We now show that the statement (\ref{as55}) in S-Condition 4 can be
equivalently simplified as
\begin{equation}
\left\{\begin{array}{ll}b\in \mathcal{R}(B),& \textnormal{if}~B\succeq 0~
\textnormal{or}~B\preceq 0,\\B\zeta+b=0,& \textnormal{otherwise}. \end{array}\right.
\label{bB}
\end{equation}
Notice that (\ref{bB}) trivially implies (\ref{as55}), so it is
sufficient to prove the converse.

Suppose $B\succeq 0$ or $B\preceq0$. Then, $x^TBx=0
~\Longleftrightarrow~ Bx=0$ and (\ref{as55}) can be recast as
\[
Bx=0 ~\Longrightarrow ~ b^Tx=0,~\forall x\in \mathbb{R}^n,
\]
which shows that $b\in \mathcal{R}(B)$.

Now assume that $B$ is indefinite. We first rewrite (\ref{as55}) by
\[
x^TBx=0 ~\Longrightarrow ~x^T(B\zeta+b)(B\zeta+b)^Tx=0,~\forall x\in \mathbb{R}^n.
\]
Since $(B\zeta+b)(B\zeta+b)^T\succeq 0$, it is further equivalent to
\[
x^TBx=0 ~\Longrightarrow ~ -x^T(B\zeta+b)(B\zeta+b)^Tx\ge 0,~\forall x\in \mathbb{R}^n.
\]
Since $B$ is indefinite, $h(x)=x^TBx$ takes both positive and
negative values. By the S-lemma with equality for homogeneous
quadratic forms, %$there is a $\mu\in\Bbb R $ such that
\begin{equation}
(\exists \mu\in\Bbb R)~~~-(B\zeta+b)(B\zeta+b)^T+\mu B  \succeq 0.\label{BbB}
\end{equation}
%If there is a $\zeta$ such that  $B\zeta+b =0$, then we can set  $\mu=0$. In this case, d is in the range of B.
Since $\mu B \succeq (B\zeta+b)(B\zeta+b)^T \succeq 0$ and $B$ is
indefinite, it must be $\mu=0$ and thus $(B\zeta+b)(B\zeta+b)^T=0$.
It implies that $B\zeta+b=0$.

%\begin{itemize}
%\item
%\item
%\end{itemize}

\section*{Acknowledgments}
The authors are grateful to the two anonymous referees for their valuable comments.


\begin{thebibliography}{99999}

\bibitem{AW}
Anstreicher, K.M.,  Wright, M.H.:
A note on the augmented Hessian
when the reduced Hessian is semidefinite. SIAM J. Optim.  {\bf11}(1), 243--253 (2000)

\bibitem{BE07}
Beck, A.: On the Convexity of a Class of Quadratic Mappings and its Application to the
Problem of Finding the Smallest Ball Enclosing a Given Intersection of Ball, J.
Global Optim. {\bf39}, 113-126  (2007)

\bibitem{BE}  Beck, A., Eldar, Y.C.:
Strong duality in nonconvex quadratic optimization with
two quadratic constraint. SIAM J. Optim.  {\bf17}(3), 844-860 (2006)


\bibitem{Ben14}
Ben-Tal, A.,   Hertog, D.den:
Hidden conic quadratic representation of some nonconvex quadratic optimization problems. Math. Program. Ser. A.  {\bf143}, 1--9  (2014)

\bibitem{Ben}
 Ben-Tal, A.,   Teboulle, M.: Hidden convexity in some nonconvex quadratically constrained
quadratic programming. Math. Program. {\bf72},  51--63 (1996)

\bibitem{Br}
Brickman, L.:  On the field of values of a matrix. Proc. Am. Math. Soc. {\bf12}, 61--66 (1961)

\bibitem{B13}
 Burer, S.,   Anstreicher, K.M.: Second-Order-Cone Constraints for
Extended Trust-Region Subproblems.  SIAM J. Optim. {\bf23}(1), 432--451  (2013)

\bibitem{D06}
 Derinkuyu, K.,   P\i nar, M.\c{C}.:
On the S-procedure and some variants.
Math. Meth. Oper. Res.  {\bf 64}, 55--77 (2006)

\bibitem{Dines}  Dines, L.L.:
On the mapping of quadratic forms. Bull. Amer. Math. Soc. {\bf47},
494--498 (1941)

\bibitem{Fa}
Fang, S.C.,  Gao, D.Y.,  Lin, G.X.,  Sheu, R.L.,   Xing, W.: Double
well potential function and its optimization in the n-dimenstional
real space -- Part I, Math. Mech. Solids (2015)
DOI:10.1177/1081286514566704



\bibitem{Feng}
 Feng, J.M.,  Lin, G.X.,  Sheu, R.L.,   Xia, Y.:
Duality and solutions for quadratic programming over single non-homogeneous
quadratic constraint. J. Global Optim.  {\bf54}(2), 275--293 (2012)

 \bibitem{F}
 Finsler, P.: \"{U}ber das vorkommen definiter und semidefiniter Formen in scharen
quadratischer Formen. Comment. Math. Helv. {\bf9}, 188-192 (1937)

 \bibitem{Fr}
 Fradkov, A.L.,  Yakubovich, V.A.: The S-procedure and the duality relation in convex quadratic programming problems, Vestnik Leningrad. Univ. {\bf1}, 81--87 (1973) %, 155

\bibitem{H75}
 Hestenes, M.R.: Optimization Theory, John Wiley $\&$ Sons  (1975)

\bibitem{Hm}
 Hmam, H.:
Quadratic optimisation with one quadratic equality constraint.
 Electronic Warfare and Radar Division DSTO Defence Science and Technology Organisation, Australia, Report DSTO-TR-2416 (2010)

\bibitem{Horn}  Horn, R.,   Johnson, C.R.: Matrix Analysis. Cambridge University
Press, Cambridge, UK (1985)

\bibitem{H}
 Hsia, Y.,  Lin, G.X.,  Sheu, R.L.: A revisit to quadratic programming
with one inequality quadratic constraint via matrix pencil.
Pac. J. Optim. {\bf10}(3), 461--481  (2014)

\bibitem{Jerrard}
Jerrard R.L.: Lower bounds for generalized Ginzburg-Landau
functionals, SIAM J. Math. Anal. {\bf30}(4),
721--746 (1999)

\bibitem{J00}
 Jeyakumar, V.: Farkas lemma: generalizations. Encylopedia of optimization.
 {\bf2}, 87--91, Kluwer Boston, USA  (2000)

\bibitem{J09a}
 Jeyakumar, V.,  Huy, N.Q.,  Li, G.: Necessary and sufficient conditions for S-lemma and
nonconvex quadratic optimization, Optim. Eng. {\bf10}, 491--503 (2009)

\bibitem{J09b}
 Jeyakumar, V.,  Lee, G.M., Li, G.Y.: Alternative theorems for quadratic inequality systems
 and global quadratic optimization, SIAM J. Optim. {\bf20}(2), 983--1001 (2009)

\bibitem{J13}
 Jeyakumar, V.,  Li, G.Y.: Trust-Region Problems with Linear Inequality
Constraints: Exact SDP Relaxation, Global Optimality and Robust
Optimization, Math. Program. {\bf147}(1-2),  171--206 (2014)

%\bibitem{K69}
% Krein, M.G.,  Smuljan, J.L.: Plus-operators in a space with an indefinite metric. Amer. Math. Soc. Transl. {\bf85},  93--113 (1969)

%\bibitem{L79}
%\textcolor{red}{Levin, J.: Mathematical models for determining the intersections of quadric surfaces, Computer Graphics and Image Processing, {\bf11}, 73--87 (1979)}

\bibitem{Mar}
Mart\'{\i}nez-Legaz, J.E.  On Brickman's theorem. J. Convex Anal. {\bf12}, 139--143 (2005)

\bibitem{M93}  Mor\'{e}, J.J.: Generalizations of the trust region problem.
Optim. Methods Softw. {\bf2},  189--209 (1993)

\bibitem{Bong}
Nguyen, V.B.,  Sheu, R.L.,  Xia, Y.: An SDP approach for quadratic fractional problems with a two-sided quadratic constraint. Optim. Methods Softw. (2015)  DOI:10.1080/10556788.2015.1029575

\bibitem{Pa}
Palanthandalam-Madapusi, H.J.,  Pelt, T.H.V.,  Bernstein, D.S.:
Matrix pencils and existence conditions for quadratic programming with
a sign-indefinite quadratic equality constraint. J. Global Optim.  {\bf 45}(4), 533--549 (2009)

\bibitem{PT}   P\'{o}lik, I.,   Terlaky, T.:
 A Survey of the S-lemma. SIAM Rev. {\bf49}(3), 371--418 (2007)

\bibitem{Poly}   Polyak, B.T.:
Convexity of quadratic transformations and its use in control and
optimization. J. Optimiz. Theory App. {\bf99}(3), 553--583 (1998)

\bibitem{Pong}
  Pong, T.K.,  Wolkowicz, H.:
 The generalized trust region subprobelm. Comput. Optim. Appl. {\bf58}(2), 273--322 (2014)

\bibitem{shor}
 Shor, N.Z.: Quadratic optimization problems. Sov. J. Comput. Syst. Sci.  {\bf25}, 1--11 (1987)


\bibitem{SW} Sturm, J.F.,   Wolkowicz, H.: Indefinite trust region subproblems and
nonsymmetric eigenvalue perturbations. SIAM J. Optim. {\bf5},
286--313 (1995)

\bibitem{sz}  Sturm, J.F.,   Zhang, S.: On cones of nonnegtive
quadratic functions. Math. Oper. Res. {\bf28}(2),
 246--267 (2003)

\bibitem{TT}
 Tuy, H.,  Tuan, H.D.:
Generalized S-lemma and strong duality in nonconvex quadratic
programming. J. Global Optim. {\bf56}(3): 1045-1072 (2013)

\bibitem{Shu}
Wang, S.,  Xia, Y.: Strong Duality for Generalized Trust Region
Subproblem: S-Lemma with Interval Bounds. Optim. Lett. (2014) DOI:10.1007/s11590-014-0812-0


\bibitem{Fa-2}
Xia, Y.,  Sheu, R.L.,  Fang, S.C., Xing, W.: Double well potential
function and its optimization in the n-dimenstional real space --
Part II, Math. Mech. Solids (2015)
DOI:10.1177/1081286514566723

\bibitem{y71}
 Yakubovich, V.A.: S-procedure in nonlinear control theory. Vestnik
Leningrad. Univ.  {\bf1},  62--77 (1971) (in Russian)

\bibitem{y77}
 Yakubovich, V.A.: S-procedure in nonlinear control theory. Vestnik
Leningrad. Univ. {\bf 4},  73--93 (1977) (English translation)

\bibitem{yz}
 Ye, Y.,   Zhang, S.: New results on quadratic minimization. SIAM J.
Optim.  {\bf14},  245--267 (2003)

\end{thebibliography}
\end{document}